\documentclass[3p,sort&compress,11pt]{elsarticle}

\usepackage{amsmath,amsthm,amsfonts,amssymb}
\usepackage{mathtools}
\usepackage{color}
\usepackage{cite}

\usepackage{pifont}
\usepackage{natbib}
\usepackage{geometry}
\usepackage{fleqn}
\usepackage{graphicx}
\usepackage{txfonts}
\usepackage{hyperref}

\newtheorem{theorem}{Theorem}[section]
\newtheorem{lemma}{Lemma}[section]
\newtheorem{corollary}{Corollary}[section]
\newtheorem{definition}{Definition}[section]

\theoremstyle{remark}
\newtheorem{remark}{Remark}[section]

\makeatletter
\DeclareRobustCommand\widecheck[1]{{\mathpalette\@widecheck{#1}}}
\def\@widecheck#1#2{%
    \setbox\z@\hbox{\m@th$#1#2$}%
    \setbox\tw@\hbox{\m@th$#1%
       \widehat{%
          \vrule\@width\z@\@height\ht\z@
          \vrule\@height\z@\@width\wd\z@}$}%
    \dp\tw@-\ht\z@
    \@tempdima\ht\z@ \advance\@tempdima2\ht\tw@ \divide\@tempdima\thr@@
    \setbox\tw@\hbox{%
       \raise\@tempdima\hbox{\scalebox{1}[-1]{\lower\@tempdima\box
\tw@}}}%
    {\ooalign{\box\tw@ \cr \box\z@}}}
\makeatother

\DeclareSymbolFont{extra}{OML}{cmm}{m}{n}
\DeclareMathSymbol{\varrho}{\mathord}{extra}{'045}

\newcommand{\e}{^\varepsilon}
\newcommand{\eps}{{\varepsilon}}
\newcommand{\ds}{\displaystyle}
\newcommand{\I}{\mathcal{I}\e}

\renewcommand{\>}{\rangle}
\renewcommand{\a}{\alpha}
\renewcommand{\b}{\beta}

\newcommand{\cupl}{\bigcup\limits}
\newcommand{\supp}{\mathrm{supp}}
\newcommand{\suml}{\sum\limits}
\newcommand{\intl}{\int\limits}
\newcommand{\liml}{\lim\limits}

\renewcommand{\phi}{\varphi}

\newcommand{\checkD}{\widecheck{D}{_i\e}}
\newcommand{\hatD}{\widehat{D}{_i\e}}

\renewcommand{\d}{\hspace{1pt}\mathrm{d}}

\numberwithin{equation}{section}

\begin{document}

\begin{frontmatter}

\title{Neumann spectral problem in a domain with very corrugated boundary}

\author[a1]{Giuseppe Cardone}
\ead{giuseppe.cardone@unisannio.it}
\address[a1]{Department of Engineering, University of Sannio, Corso Garibaldi 107, 82100 Benevento, Italy}
\cortext[cor1]{Corresponding author}

\author[a2]{Andrii Khrabustovskyi\corref{cor1}}
\ead{andrii.khrabustovskyi@kit.edu}
\address[a2]{Department of Mathematics, Karlsruhe Institute of Technology, Kaiserstra{\ss}e 89-93, 76133 Karlsruhe, Germany}

\begin{abstract}

Let $\Omega\subset\mathbb{R}^n$ be a bounded domain. We perturb it to a domain $\Omega^\varepsilon$ attaching a family of small protuberances with "room-and-passage"-like geometry ($\varepsilon>0$ is a small parameter). Peculiar spectral properties of Neumann problems in so perturbed domains were observed for the first time by R.~Courant and D.~Hilbert. We study the case, when the number of protuberances tends to infinity as $\varepsilon\to 0$ and they are $\varepsilon$-periodically distributed along a part of $\partial\Omega$. Our goal is to describe the behaviour of the spectrum of the operator $\mathcal{A}^\varepsilon=-(\rho^\varepsilon)^{-1}\Delta_{\Omega^\varepsilon}$, where $\Delta_{\Omega^\varepsilon}$ is the Neumann Laplacian in $\Omega^\varepsilon$, and the positive function $\rho^\varepsilon$ is equal to $1$ in $\Omega$. We prove that the spectrum of $\mathcal{A}^\varepsilon$ converges as $\varepsilon\to 0$ to the "spectrum" of a certain boundary value problem for the Neumann Laplacian in $\Omega$ with boundary conditions containing the spectral parameter in a nonlinear manner. Its eigenvalues may accumulate to a finite point.

\end{abstract}


\journal{a journal}
\end{frontmatter}

\section{Introduction}

Let $\Omega$ be a fixed domain in $\mathbb{R}^n$. We perturb it to a family of 
domains $\{\Omega\e\subset\mathbb{R}^n\}_{\eps}$, here $\eps>0$ is a small 
parameter.
{\color{black}It is well known that if the perturbation is regular enough 
(see, e.g., \citep[Chapter~VI, \S~2.6]{CH} for more precise 
statement), then the $k$-th eigenvalue of the Neumann Laplacian in $\Omega\e$ 
converges to the $k$-th eigenvalue of the Neumann Laplacian in $\Omega$ (the 
same also true for the Dirichlet or mixed boundary conditions). 
In general, however, this is not true -- 
even if $\Omega\e$ differs from $\Omega$ only in a ball of the radius 
$\mathcal{O}(\eps)$
as it is 
evident from the following example going back to \textit{R.~Courant and  
\textcolor{black}{D.~Hilbert}}.}
Let $\Omega$ be a unit square $\mathbb{R}^2$.
We perturb $\Omega$ to a domain $\Omega\e$ attaching to it a small domain, which consists of a square $B\e$ ("room") with a side length $b\e$ and a narrow rectangle $T\e$ ("passage") with side lengths $d\e$ and $h\e$ --- see Fig. \ref{fig0} (left picture):
$$\Omega\e=\Omega\cup\left(B\e\cup T\e\right).$$
We denote by $\Delta_{\Omega}$ and $\Delta_{\Omega\e}$ the Neumann Laplacians in $\Omega$ and $\Omega\e$, correspondingly. The first eigenvalues of both $-\Delta_{\Omega}$ and $-\Delta_{\Omega\e}$ are zero. The second eigenvalue of $-\Delta_{\Omega}$ is strictly positive, while it was shown in \citep[Chapter~VI, \S~2.6]{CH} that the second eigenvalue of $-\Delta_{\Omega\e}$ tends to zero as $\eps\to 0$ provided $d\e=\eps^4$, $b\e=h\e=\eps$.  

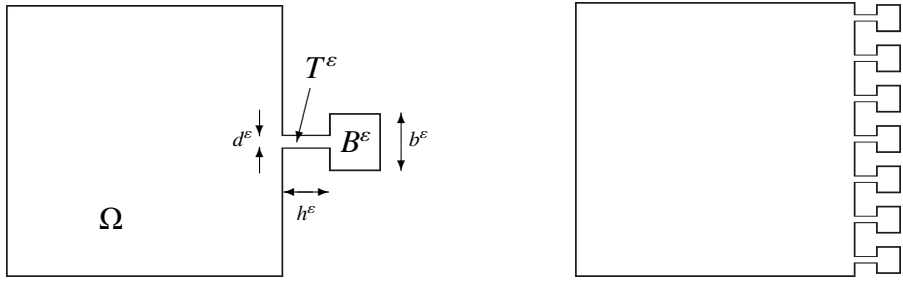
\begin{figure}[h]\large
\begin{center}
\begin{picture}(320,110)
\scalebox{0.095}{ \includegraphics{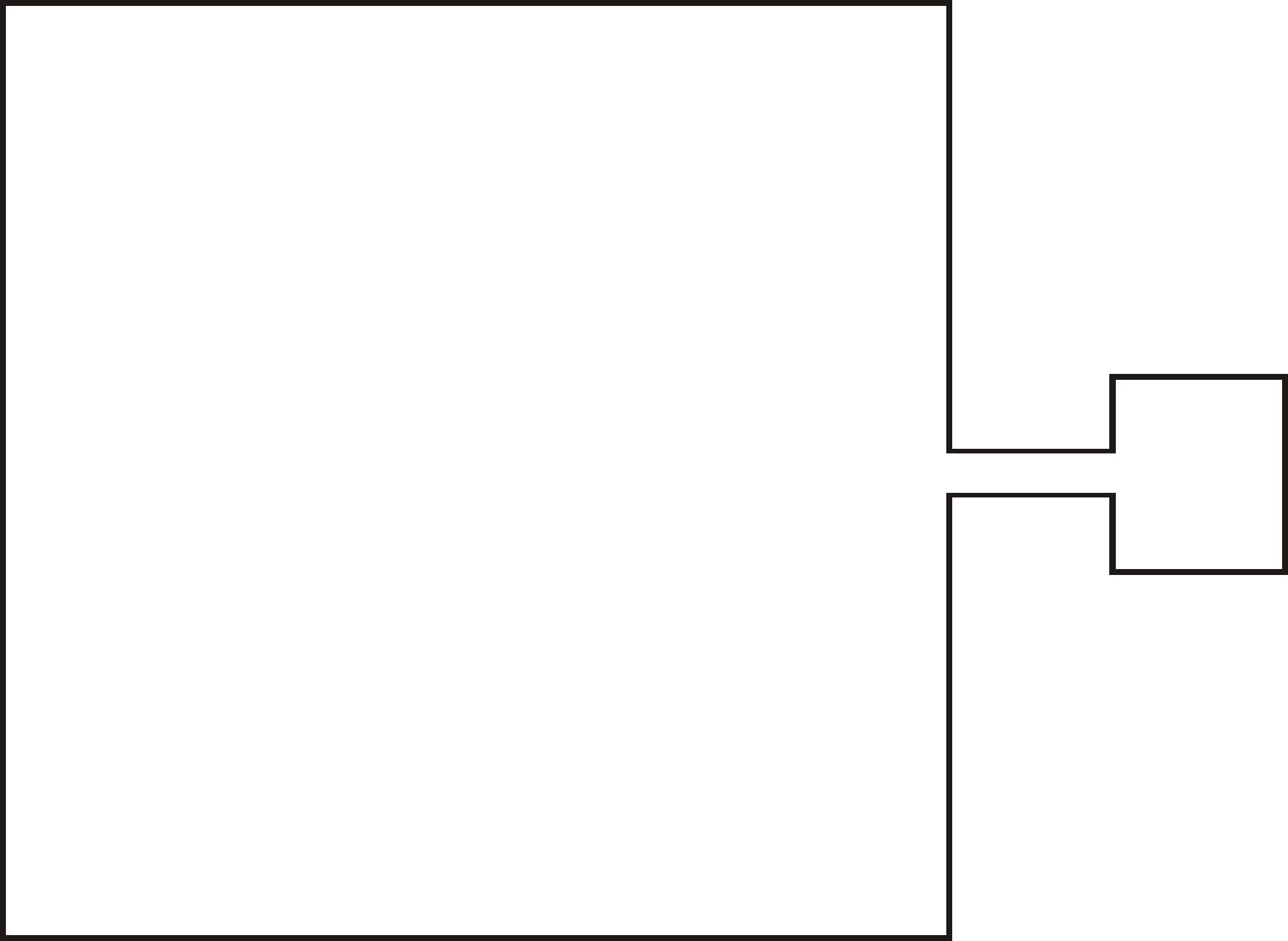}}\qquad\qquad\qquad
\scalebox{0.300}{ \includegraphics{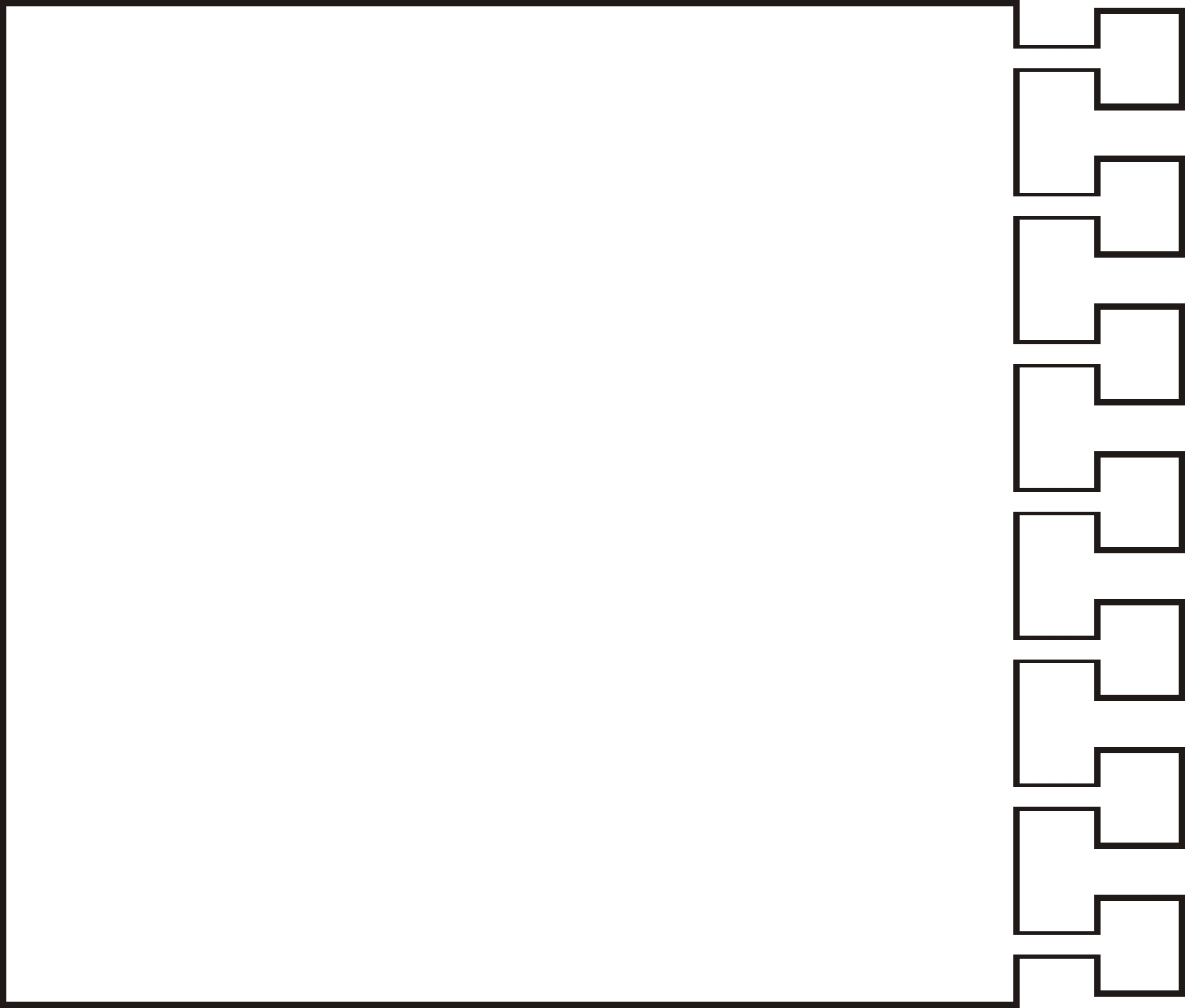}}

\put(-300, 18){$\Omega$} 

\put(-223, 75){$T\e$} \put(-221,71){\vector(-1,-4){5}}

\put(-210, 47){$B\e$} 

\put(-250,45){$^{d\e}$}  \put(-240,39){\vector(0,1){10}} \put(-240,63){\vector(0,-1){10}}

\put(-184,45){$^{b\e}$} \put(-187,50){\vector(0,1){12}} \put(-187,50){\vector(0,-1){10}} 

\put(-226,17){$^{h\e}$} \put(-226,32){\vector(1,0){12}} \put(-220,32){\vector(-1,0){11}}

\end{picture}
\end{center}

\caption{Domains with attached "room-and-passage"s}\label{fig0}
\end{figure}

Later this example was studied in more detail by \textit{J.M.~Arrieta, 
\textcolor{black}{J.K.~Hale} and Q.~Han} \citep{AHH} for more general geometry of 
"rooms" and "passages". 
Taking almost the same 
ratios between the "room" diameter, the "passage"
cross-section diameter and the "passage" length as those ones in \citep{CH} they proved that for $k\geq 2$ the $k$-th eigenvalue $\lambda_k\e$ of $-\Delta_{\Omega\e}$ converges to the $(k-1)$-th eigenvalue $\lambda_{k-1}$ of $-\Delta_{\Omega}$ as $\eps\to 0$. Also they generalized this result to the case of 
\textit{finitely} many attached "room-and-passage"-like domains proving that  
$\liml_{\eps\to 0}\lambda_k\e=0$ as $k=2,\dots,r+1$ and $\liml_{\eps\to 0}\lambda_k\e=\lambda_{k-r}$ as $k\geq r+2$ (here $r\in\mathbb{N}$ is the number of attached domains).

\textit{E.~Sanchez-Palencia} in his book \citep{Sanchez} (see Chapter XII, \S 4]) considered the case, when $\Omega\e$ is obtained by attaching \textit{several} "room-and-passage"-like domains, whose number goes to infinity as $\eps\to 0$ (Fig.~\ref{fig0}, right picture). He considered the "rooms" and "passages" of the same size as those ones in \citep{CH} and proved that 
\begin{gather}\label{hb+}
\text{for any }\lambda\in
\sigma(-\Delta_{\Omega})\text{ there exists }\lambda\e\in\sigma(-\Delta_{\Omega\e})\text{ such
that }\liml_{\eps\to 0}\lambda\e=\lambda. 
\end{gather}
\textcolor{black}{Hereinafter} by $\sigma(\cdot)$ we denote the spectrum of an 
operator.
Also he shown that, similarly to the case of one attached "room-and-passage"-like domain, the second eigenvalue of $-\Delta_{\Omega\e}$ goes to zero.\smallskip

\textit{The goal of the present work} is to extend the results obtained in \citep{Sanchez} 
considering various cases of sizes of "rooms" and "passages" and perturbing mass density in the "rooms". Below we sketch the main results of this work.\smallskip

Let $\Omega$ be a bounded domain in $\mathbb{R}^n$ ($n\geq 2$). It is supposed that the part of $\partial\Omega$ belongs to a $(n-1)$-dimensional hyperplane. We denote this part of $\partial\Omega$ by $\Gamma$. Let $\eps>0$ be a small parameter and $b\e$, $d\e$, $h\e$ be positive numbers going to zero as $\eps\to 0$, $d\e\leq b\e\leq\eps$; also we suppose that $d\e$ tends to zero not too fast, namely
\begin{gather}
\label{>>}d\e\gg \exp(-a/\eps),\forall a>0\ (\text{if }n=2)
\text{\quad or\quad }
d\e\gg \eps^{n-1\over n-2}\ (\text{if }n>2)
\end{gather}
as $\eps\to 0$.

We attach a family of "room-and-passage"-like domains $\eps$-periodically along 
$\Gamma$. Each
attached domain  consists of two building blocks:
\begin{itemize}

\item the "room" $\simeq\ b\e B$, where $B$ is a fixed domain in $\mathbb{R}^{n}$,

\item the "passage" $\simeq\ d\e D\ \times [0,h\e]$, where $D$ is a fixed domain in $\mathbb{R}^{n-1}$. 

\end{itemize}
We denote these "rooms" and "passages" by $B_i\e$ and $T_i\e$ , correspondingly (the parameter $i$ counts them). The total number $N(\eps)$ of "rooms" (or "passages") tends to infinity as $\eps\to 0$, namely 
\begin{gather}\label{total}
N(\eps)\sim { \eps^{1-n} |\Gamma|}.
\end{gather}
Hereinafter we use the same notation $|\cdot|$ either for
the volume of a domain in $\mathbb{R}^n$ (for example, $|B|$), for
the volume of a domain in $\mathbb{R}^{n-1}$ (for example, $|D|$) or for
the area of an $(n-1)$-dimensional hypersurface in $\mathbb{R}^{n}$ (for example, $|\Gamma|$). 

We impose also some additional conditions (see \eqref{ass11}-\eqref{ass14}) guaranteeing that the neighbouring "rooms" are pairwise disjoint, that the $i$-th "room" and the $i$-th "passage" are correctly glued, and that the distance between the neighbouring "passages" is not too small, namely for $i\not=j$
 $\mathrm{dist}(\textcolor{black}{T_i\e,T_j\e})\geq C\eps$  (here $0<C<1$).

\begin{figure}[h]
\begin{center}
 \begin{picture}(220,140)
\scalebox{0.45}{ \includegraphics{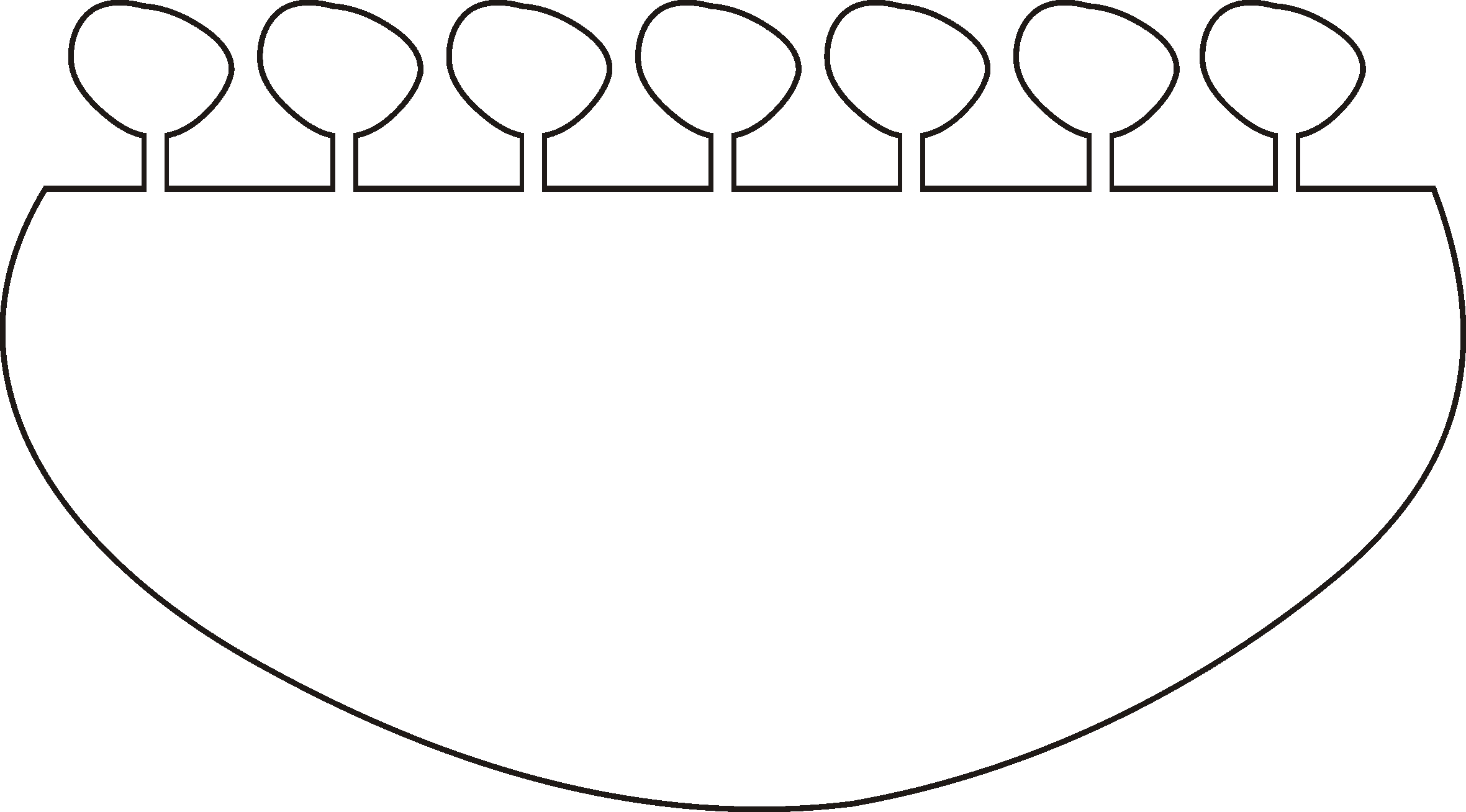}}

\put(-70, 45){$\Omega$} 

\put(0, 117){$\Gamma$} \put(-1,116){\vector(-2,-1){15}}

\put(-258, 112){$T_i\e$} \put(-244,115){\vector(1,0){15}}

\put(-235, 127){$B_i\e$} 

\put(-183, 95){$\eps$}\put(-183,98){\vector(-1,0){14}} 
\put(-177,98){\vector( 1,0){14}} 

\end{picture}
\end{center}

\caption{The domain $\Omega\e$}\label{fig1}
\end{figure}

We denote by $\Omega\e$ the obtained domain (see Fig.\ref{fig1}),
$$\Omega\e=\Omega\cup\left( \cupl_i (T_i\e\cup B_i\e)\right)$$
and introduce the operator
\begin{gather*}
\mathcal{A}\e=-{1\over\rho\e}\Delta_{\Omega\e}.
\end{gather*}
Here $\Delta_{\Omega\e}$ is the Neumann Laplacian in $\Omega\e$, the function $\rho\e$ is equal to $1$ everywhere except  the union of the "rooms", where it is equal to the constant $\varrho\e>0$. 
The operator $\mathcal{A}\e$ describes vibrations of the medium occupying 
$\textcolor{black}{\Omega\e}$ and having the mass density
$\rho\e$. 

Our goal is to study the behaviour of the spectrum $\sigma(\mathcal{A}\e)$ as $\eps\to 0$ 
under the assumption that the following limits exist:
\begin{gather}\label{qr_intro}
\liml_{\eps\to 0}{(d\e)^{n-1} |D|\over h\e \varrho\e(b\e)^n |B|}=:q\in [0,\infty],\quad \liml_{\eps\to 0}{\varrho\e(b\e)^n |B|\over\eps^{n-1}}=: r\in [0,\infty).
\end{gather}
We note, that the first limit is allowed to be infinite. 
The finiteness of $r$ implies the uniform (with respect to $\eps$) boundedness
of the total mass $m\e_B$ of the "rooms", namely, using \eqref{total}, we obtain:
\begin{gather*}
m\e_B:=\intl_{\cupl_i B_i\e}\rho\e \d x=\varrho\e \suml_{i}|B_i\e|=\varrho\e(b\e)^n |B| N(\eps)\sim
{\varrho\e(b\e)^n |B|\cdot |\Gamma|\over\eps^{n-1}}\sim r|\Gamma|\text{ as }\eps\to 0.
\end{gather*}
{\color{black}The meaning of the parameter $q$ will be explained in the end of 
Section \ref{sec1} (see Remark \ref{remark_q}).} 

Despite the fact that our problem contains many parameters the form 
of the limit spectral problem depends essentially only on either $q$ is finite or infinite and $r$ is positive or zero. 

We present the results in a rather formal way, more precise statements are 
formulated in the next section using the operator theory language. 
{\color{black}In what follows we study the behavior of spectra on finite 
intervals, the convergence is understood in the Hausdorff sense, }
see Definition \ref{def1}. One has the following four cases:
\begin{itemize}
\item[1)] $q<\infty,\ r>0$. In this case $\sigma(\mathcal{A}\e)$ converges as $\eps\to 0$ to the union of the point $q$ and the set of eigenvalues of the spectral problem
\begin{gather}\label{formal1}
\begin{cases}
-\Delta u=\lambda u&\text{ in }\Omega,\\ 
{\partial u\over\partial  n}={\lambda qr\over q-\lambda} u&\text{ on }\Gamma,\\
{\partial u\over\partial  n}=0&\text{ on }\partial\Omega\setminus\Gamma,
\end{cases}
\end{gather}
where $n$ is the outward-pointing unit normal to $\partial\Omega$ 
($\lambda$ is an eigenvalue if \eqref{formal1} has a non-trivial solution $u$).

The set of eigenvalues of the problem \eqref{formal1} consists of two ascending sequences --- one of them goes to infinity and the second one goes to $q$ (in the case $q=0$ the second sequence disappears). 

\item[2)] $q<\infty,\ r=0$. In this case $\sigma(\mathcal{A}\e)$ converges to the set $\sigma(-\Delta_\Omega)\cup \{q\}$.  

Formally, this result follows from the previous one if we set $r=0$ in \eqref{formal1}. 
But, since the spectral properties of \eqref{formal1} change drastically under the passage from $r>0$ to $r=0$, we  write out these cases separately. 

\item[3)] $q=\infty,\ r>0$. In this case $\sigma(\mathcal{A}\e)$ converges to the set of  eigenvalues of the spectral problem
\begin{gather}\label{formal2}
\begin{cases}
-\Delta u=\lambda u&\text{ in }\Omega,\\ 
{\partial u\over\partial  n}=\lambda r u&\text{ on }\Gamma,\\
{\partial u\over\partial  n}=0&\text{ on }\partial\Omega\setminus\Gamma.
\end{cases}
\end{gather}

\item[4)] $q=\infty,\ r=0$. In this case $\sigma(\mathcal{A}\e)$ converges to $\sigma(-\Delta_\Omega)$.

\end{itemize}

Obviously, all these cases can be realized. For example, if we take 
$$n=2,\quad d\e=\eps^\a\ (\a\geq 1),\quad b\e=h\e=\eps,\quad \varrho\e=\eps^{\b}\ (\b\geq -1)$$ then
condition \eqref{>>} holds true and the limits \eqref{qr_intro} exist, namely
\begin{gather*}
\begin{cases}
r>0,&\text{ if }\b=-1,\\
r=0,&\text{ if }\b>-1,
\end{cases}
\text{\quad and \quad}
\begin{cases}
q>0,&\text{ if }\a=\b+3,\\
q=0,&\text{ if }\a>\b+3,\\
q=\infty,&\text{ if }\a<\b+3.
\end{cases}
\end{gather*}

{\color{black}
In the current paper the convergence of spectra is understood as convergence of
sets (namely, Hausdorff convergence), and not in terms of  ``spectral 
convergence'' of the corresponding operators. In particular, 
the behaviour as $\eps\to 0$ of the $k$-th eigenvalue $\lambda_k\e$ of $\mathcal{A}\e$ for fixed $k$ is not analyzed. We will study this question elsewhere. In the current paper we only prove that
$$\lambda_k\e\to 0\text{ as }\eps\to 0\text{ provided }q=0.$$
More generally, we show that
\begin{gather}\label{tr}
\sup\limits_{k\in\mathbb{N}}\left(\underset{\eps\to 0}{\overline{\lim}}\lambda_k\e\right)\leq q.
\end{gather}
Following \citep{Melnyk1} we call the quantity staying in the left-hand-side of \eqref{tr} \textit{the threshold of low eigenfrequencies}.}

Note, that the choice of the boundary conditions on \textit{unperturbed} part of 
the boundary (i.e. on $\partial\Omega\setminus\Gamma$) is inessential -- instead 
of the Neumann boundary conditions we can prescribe, for example, the Dirichlet 
or mixed ones. These conditions will be inherited by the limit spectral problem.

The paper is organized as follows. In Section \ref{sec1} we set up the problem and formulate the main results absorbed in Theorem \ref{th1}, also in Section \ref{sec1} we prove inequality \eqref{tr} (Theorem \ref{th2}). Theorem \ref{th1} is proved in Section \ref{sec2}: in Subsection \ref{subsec21} we establish some auxiliary estimates, Section \ref{subsec22} is devoted to the case $q<\infty$, the case $q=\infty$ is treated in Subsection \ref{subsec23}. 
\smallskip

In the end of the introduction we would like to make some bibliographical comments.\smallskip

\noindent\textbf{1.} In the "classical" case $\varrho\e= 1$ (i.e. $\mathcal{A}\e=-\Delta_{\Omega\e}$) the property \eqref{hb+} follows from the next general result obtained in \citep{LS}:
let $\Omega\subset\mathbb{R}^n$ be a fixed domain and $\{\Omega\e\subset\mathbb{R}^n\}_{\eps}$ be a family of domains satisfying some mild regularity assumptions and 
\begin{gather}\label{LS}
\Omega\subset\Omega\e,\quad \mathrm{meas}(\Omega\e\setminus\Omega)\to 0\text{ as }\eps\to 0
\end{gather}
(here $\mathrm{meas}(\cdot)$ stays for the Lebesgue measure in $\mathbb{R}^n$), then \eqref{hb+} holds true. And indeed, although the number of attached "rooms" and "passages" tends to infinity, their total measure tends to zero.

It may happen, however, that \eqref{LS} holds true, but there exists a sequence $\lambda\e\in \sigma(-\Delta_{\Omega\e})$ converging to a point not belonging to $\sigma(-\Delta_{\Omega})$. As we will see below this can happen for "room-and-passage"-like perturbations. Another important class of such domains are so-called dumbbell-shaped domains. In a simplest case they are defined as follows: let $\Omega$ be a union of two disjoint domains $\Omega_j$, $j=1,2$ and $\Omega\e=\Omega\cup T\e$, where $T\e$ is a narrow straight channel connecting $\Omega_1$ and $\Omega_2$ and approaching a $1$-dimensional line segment of the length $L$ as $\eps\to 0$. It can be proved that if $\sigma(-\Delta_{\Omega\e})\ni\lambda\e\to\lambda$ as $\eps\to 0$ then either $\lambda\in \sigma(-\Delta_{\Omega_1})\cup \sigma(-\Delta_{\Omega_2})$ or $\lambda=\left({\pi k\over L}\right)^2$ for some $k\in\mathbb{N}$.
The spectral properties of boundary value problems posed in dumbbell-shaped domains were studied in a lot of papers -- see, e.g., 
\citep{Arrieta1,Jimbo}. Some general results allowing to characterize the set of accumulation points, which are not in the spectrum of $\sigma(-\Delta_{\Omega})$, were obtained in \citep{Arrieta2}.
\smallskip

\noindent\textbf{2.} One can also study the behaviour of the spectrum of the Dirichlet Laplacian under a perturbation of the boundary of a domain. But in this case the continuity of eigenvalues holds for rather wide set of perturbations. For example (cf. \citep{RT}), if   $\Omega$ is a bounded domains and
\begin{gather}\label{D}
\begin{array}{l}
\text{for every compact set $F\subset\Omega$  there exists $\eps_F>0$ such that $F\subset\Omega\e$ provided $\eps<\eps_F$,}\\
\text{for every open set $O\supset\overline{\Omega}$ there exists $\eps_O>0$ such that $\Omega\e\subset O$ provided $\eps<\eps_O$},
\end{array}
\end{gather}
then the $k$-th eigenvalue of the Dirichlet Laplacian in $\Omega\e$ converges to the $k$-th eigenvalue of the Dirichlet Laplacian in $\Omega$. It is easy to see that "room-and-passage"-like perturbations described above (with $\rho\e\equiv 1$) satisfy conditions \eqref{D}. However, the situation might be more complicated if together with the geometry of a boundary we perturb the mass density near it.\smallskip

\noindent\textbf{3.}
Boundary value problems in domains with rapidly oscillating boundary
attract a great attention of mathematicians in recent years. 
Such problems are motivated by various applications in physics and 
engineering sciences (for example, in scattering of acoustic and electromagnetic waves on small periodic obstacles). We mention here some papers devoted to such problems -- 
\citep{Belyaev,Brizzi,MN,CFP}, more references one can find in \citep{ACG}.
To the best of our knowledge, problems in domains with rapidly oscillating boundary and 
"room-and-passage"-like geometry of a period cell were considered only in the book  \citep{Sanchez} mentioned above.
\smallskip

\noindent\textbf{4.}
The asymptotic behaviour of eigenvibrations of a body with 
mass density singularly perturbed near its boundary was studied, for example, in 
\citep{Lobo1,Lobo2,CCDP,NP} and many other papers.  
The overview of results in this area one can find in the introduction of \citep{Che1}. 
Perturbations involving both the perturbation of the boundary and of the mass density were studied in \citep{Melnyk1,Melnyk2}, here the domain in $\mathbb{R}^2$ is perturbed by attaching a lot of narrow strips of a fixed length (so-called \textit{thick junctions}), on these strips the mass density is large. 
\smallskip

\noindent\textbf{5.} As we announced above our limit spectral problem may 
contain the spectral parameter in boundary conditions. Namely, we get the 
following boundary conditions: 
\begin{gather}\label{meromo}
{\partial u\over\partial n}=\mathcal{F}(\lambda) u,
\end{gather}
where $\mathcal{F}(\lambda)$ is either linear ($q=\infty$) or rational 
($q<\infty$) function. In the later case $\mathcal{F}(\lambda)$ has exactly one 
pole (the point $q$), which is a point of accumulation of eigenvalues.

The boundary conditions of the form \eqref{meromo} appear in some problems with 
concentrated masses (cf. \citep{Lobo1}), also \textcolor{black}{the same effect 
can be observed} 
in problems involving thick junctions -- see, e.g., \citep{MN}. In these 
problems $\mathcal{F}$ is a meromorphic function with a sequence of poles. 

Elliptic boundary value problems with boundary conditions containing a spectral parameter were studied in many papers -- see, e.g.,  \citep{Behrndt} and references  therein.
\smallskip

\noindent\textbf{6.}
Domains with "room-and-passage"-like geometry are widely used in order to 
construct examples illustrating various phenomena in Sobolev spaces theory (see, for example, \citep{Mazya}). For instance, it is well-known, that if $\Omega$ is a \textit{bounded} domain then the imbedding $i_\Omega: H^1(\Omega)\hookrightarrow L_2(\Omega)$
is compact provided $\partial\Omega$ is sufficiently regular.  If $\partial\Omega$ is not smooth, then $i_\Omega$ need not be compact. One of the examples demonstrating this can be constructed from a sequence of "rooms", which are joint by a sequence of "passages". 
The number of the  "rooms" and "passages" is infinite, but their diameters 
decrease in such a way that their union is bounded. It turns out 
(cf. \citep{Fraenkel}) that under a special choice of sizes of "rooms" and "passages" the embedding $i_\Omega$ is non-compact. 

The non-compactness  of  $i_\Omega$ leads to occurrence of 
essential spectrum for $\Delta_\Omega$. In this connection, we mention the following nice result from \citep{HSS}: for an arbitrary closed set $S\subset [0,\infty)$, $\{0\}\in S$ and $n\geq 2$ one can construct the domain $\Omega\subset\mathbb{R}^{n}$ such that the essential spectrum of $-\Delta_\Omega$ is just
this set $S$. In their constructions the authors of \citep{HSS} used "room-and-passage"-like domains.

\smallskip
\noindent\textbf{7.} One can consider also the "bulk" analogue of our domain 
$\Omega\e$. 
Namely, we perturbed the domain $\Omega$ to the family $\{\Omega\e\}_\eps$,  $\Omega\e=\left(\Omega\setminus\cup_{i}D_i\e\right)\cup(\cup_i B_i\e)$ of $n$-dimensional Riemannian manifolds, where $\{D_i\e\}_i$ is a family of small holes distributed $\eps$-periodically in $\Omega$ (they play the role of "passages"), $\{B_i\e\}_i$ is a family of spherical surfaces (they play the role of "rooms"), $B_i\e$ is glued to the boundary of $D_i\e$.  When $\eps\to 0$ the number of attached surfaces goes to infinity, while their radii goes to zero. Instead of the usual Laplacian we study the Lalplace-Beltrami operator $\widetilde\Delta_{\Omega\e}$ in $\Omega\e$ (the choice of the boundary conditions on $\partial\Omega\e=\partial\Omega$ is inessential). Evolution equations involving  $\widetilde\Delta_{\Omega\e}$ were studied in \citep{BK,BCK}.
The behaviour of its spectrum as $\eps\to 0$ was studied in \citep{Khrab1,Khrab2,Khrab3}. 
In particular, it was shown in \citep{Khrab1} that under a suitable choice of sizes of $D_i\e$ and $B_i\e$
the limit spectral problem in $\Omega$ has the form $-\Delta u=\mathcal{F}(\lambda)u$, where 
$\mathcal{F}(\lambda)$ is  a rational function with one pole; the spectrum of this problem has the same structure as we have in our limit problem, when $q<\infty$, $r>0$.  

{\color{black}
\smallskip
\noindent\textbf{8.} 
Our research is inspired, in particular, by spectral problems involving
periodic differential operators posed in domains with periodic waveguide-like 
geometry (tubes with periodically varying cross-section, etc.). It is well-known 
that spectra of such operators have band structure, that is a locally finite 
union of 
compact
intervals -- \textit{bands}. In general the bands may overlap, otherwise we 
have a \textit{gap} in the spectrum, i.e. an open bounded interval having an 
empty 
intersection with the spectrum, but with edges belonging to it.
The 
presence of gaps is important for the description of wave processes 
which are governed by differential operators under consideration. If the 
wave frequency belongs to a gap, then the corresponding wave cannot propagate in 
the medium. 
The problem of constructing of periodic differential operators posed 
in domains with periodic waveguide-like geometry and having non-void spectral 
gaps attracts a lot attention in the last years -- see, e.g., 
\citep{BorPan,Cardone1,Cardone2}. 

Our guess is that ``room-and-passage''-like perturbation of a straight 
waveguide 
will imply the opening of gaps. 
This conjecture is based on the following observation.
Let $\Omega=(0,d)\times (-l,l)$  ($d,l>0$) be a strip in $\mathbb{R}^2$, 
$\Gamma=\{d\}\times (-l,l)$. It is not hard to show (using the proof of Lemma 
\ref{lm1}) that there exists $\widehat q\in (q,\infty)$ 
such that for an arbitrary $l$ the interval $(q,\widehat q)$ has an empty 
intersection with the 
spectrum of the problem \eqref{formal1}
provided $d$ is less some $d_0$ (where $d_0$ is 
independent of 
$l$). This suggests that the spectrum of the operator $\mathcal{A}\e$ posed in 
the unbounded ($l=\infty$) periodic waveguide-like domain $\Omega\e$ has a gap 
provided 
$\eps$ and the thickness of unperturbed strip $\Omega$ are small enough. This 
problem will be studied in all details in our subsequent work.
}

\section{\label{sec1}Setting of the problem and main results}

In what follows by $x'=(x_1,\dots,x_{n-1})$ and $x=(x',x_n)$ we denote the Cartesian coordinates in $\mathbb{R}^{n-1}$ and $\mathbb{R}^{n}$, correspondingly. 

Let $\Omega$ be an open bounded domain in $\mathbb{R}^{n}$ ($n\geq 2$) with a Lipschitz boundary and satisfying the condition $\Omega\subset \{x\in\mathbb{R}^n:\ x_n<0\}$. We denote
$$\Gamma=\partial\Omega\cap \{x\in\mathbb{R}^n:\ x_n=0\}.$$ 
It is supposed that the set $\left\{x'\in\mathbb{R}^{n-1}:\ (x',0)\in\Gamma\right\}\subset\mathbb{R}^{n-1}$ has non-empty interior.

In the space $L_2(\Omega)$  we introduce the sesquilinear form $\eta$ defined by the formula
\begin{gather*}
\eta[u,v]=\intl_{\Omega}\nabla u\cdot\nabla\bar {v} \d x,
\end{gather*}
with $\mathrm{dom}(\eta)=H^1(\Omega).$ The form  $\eta$ is densely defined, closed, positive and symmetric, whence (cf. \citep[Theorem VIII.15]{Reed}) there
exists the unique self-adjoint and positive operator
$\mathcal{A}$ associated with the form
$\eta$, i.e.
\begin{gather*}
(\mathcal{A} u,v)_{L_2(\Omega)}=
\eta[u,v],\quad\forall u\in
\mathrm{dom}(\mathcal{A}),\ \forall v\in
\mathrm{dom}(\eta).
\end{gather*}
The operator $\mathcal{A}$ is the Laplacian in $\Omega$ subject to the Neumann boundary conditions on $\partial\Omega$.

Now, we introduce the "rooms" and "passages". Let $\eps>0$ be a small parameter. We denote by $\{x^{i,\eps}\}_{i\in\mathbb{Z}^{n-1}}$ the family of points $\eps$-periodically distributed on the hyperplane $\{x\in\mathbb{R}^n:\ x_n=0\}$:
$$x^{i,\eps}=(\eps i,0),\ i\in\mathbb{Z}^{n-1}.$$
For $i\in\mathbb{Z}^{n-1}$ we set:
\begin{itemize}

\item[] $B_i\e=\left\{x\in\mathbb{R}^n:\ \ds{1\over b\e}\left(x-\tilde x^{i,\eps}\right)\in B,\text{ where }\tilde x^{i,\eps}=x^{i,\eps}+(0,h\e)\right\}$\quad (the $i$-th "room"),

\item[] $T_i\e=\left\{x\in\mathbb{R}^n:\ \ds{1\over d\e}\left(x'-\eps i\right)\in D,\ 0\leq x_n\leq h\e\right\}$\quad (the $i$-th "passage"),

\end{itemize}
where $b\e$, $d\e$, $h\e$ are positive constants,
$B$ and $D$ are open bounded domains in $\mathbb{R}^{n}$ and $\mathbb{R}^{n-1}$, correspondingly, having Lipschitz boundaries and satisfying the conditions
\begin{gather}\label{ass11}
B\subset \left\{x\in\mathbb{R}^n:\ x'\in\left(-{1\over 2},{1\over 2}\right)^{n-1},\ x_n>0\right\},\\\label{ass12}
\exists R\in \left(0,{1\over 2}\right):\ \left\{x\in\mathbb{R}^n:\ |x'|<R,\ x_n=0\right\}\subset \partial B,\\
\label{ass13}
\{0\}\in D\subset \left\{x'\in\mathbb{R}^{n-1}:\ |x'|<R\right\},\text{ where }R\text{ comes from }\eqref{ass12},\\\label{ass14}
d\e\leq b\e\leq \eps,\\
\label{ass15}
h\e\to 0\text{ as }\eps\to 0.
\end{gather}
Conditions \eqref{ass11}-\eqref{ass14} imply that the neighbouring "rooms" are pairwise disjoint
and guarantee correct gluing of the $i$-th "room" and the $i$-th "passage" (namely, the upper face of $T_i\e$ is contained in $\partial B_i\e$).  Also, it follows from \eqref{ass13}-\eqref{ass14} that
the distance between the neighbouring "passages" is not too small, namely
for $i\not=j$ one has $\mathrm{dist}(\textcolor{black}{T_i\e,T_j\e})\geq \eps-2R 
d\e\geq 2\eps\left ({1\over 2}-R\right)$.

Additionally, we  suppose that
\begin{gather}\label{ass2} 
\eps^{n-1}\mathbf{D}\e\to 0\text{ as }\eps\to 0,
\end{gather}
where 
\begin{gather}\label{De}
\mathbf{D}\e=\begin{cases}
|\ln d\e| ,& n=2,\\
(d\e)^{2-n},& n>2
\end{cases}
\end{gather}
(obviously, \eqref{ass2} is equivalent to \eqref{>>}).

Attaching the "rooms" and the "passages" to $\Omega$ we obtain the perturbed 
domain
$$\Omega\e=\Omega\cup \left({\cupl_{i\in \I}\left(T_i\e\cup B_i\e\right)}\right),$$
where 
$$\mathcal{I}\e=\left\{i\in\mathbb{Z}^{n-1}:\ x^{i,\eps}\in\Gamma\text{ and }\mathrm{dist}(x^{i,\eps},\partial\Omega\setminus\Gamma)\geq \eps {\sqrt{n}\over 2}\right\}.$$ The domain $\Omega\e$ is depicted on Fig. \ref{fig1}.

Now, let us define accurately the operator $\mathcal{A}\e$, which will be the main object of our interest. We denote by $\mathcal{H}\e$ the Hilbert space of functions from $L_2(\Omega\e)$ endowed with a scalar product
\begin{gather*}
(u,v)_{\mathcal{H}\e}=\intl_{\Omega\e} u(x)\overline{v(x)}  \rho \e(x) \d x,
\end{gather*}
where the function $\rho\e(x)$ is defined as follows:
$$\rho\e(x)=\begin{cases}1,&x\in\Omega\cup\left(\cupl_{i\in\I}T_i\e\right),\\\varrho\e,& x\in \cupl_{i\in\I}B_i\e,\end{cases}\quad \varrho\e>0\text{ is a constant.}
$$
By $\eta\e$ we denote the sesquilinear form in $\mathcal{H}\e$ defined by the formula
\begin{gather*}
\eta\e[u,v]=\intl_{\Omega\e} \nabla u\cdot\nabla \bar v \d x
\end{gather*}
with 
$\mathrm{dom}(\eta\e)=H^1(\Omega\e).$ 
The form  $\eta\e$ is densely defined, closed, positive and symmetric. We denote by
$\mathcal{A}\e$ the operator associated with this form, i.e.
\begin{gather*}
(\mathcal{A}\e u,v)_{\mathcal{H}\e}=
\eta\e[u,v],\quad\forall u\in
\mathrm{dom}(\mathcal{A}\e),\ \forall v\in
\mathrm{dom}(\eta\e).
\end{gather*}
In other words, the operator $\mathcal{A}\e$ is defined by the operation $-{1\over \rho\e}\Delta$ in $\Omega\e$ and
the  Neumann boundary conditions on $\partial\Omega\e$.
 
The spectrum $\sigma(\mathcal{A}\e)$ of the operator $\mathcal{A}\e$ is purely discrete. The goal of this work is to describe the behaviour of $\sigma(\mathcal{A}\e)$ as $\eps\to 0$  
under the assumption that the following limits exists:
\begin{gather}\label{qere+}
q:=\liml_{\eps\to 0}q\e,\ r:=\liml_{\eps\to 0}r\e,\quad q\in [0,\infty],\ r\in [0,\infty),
\end{gather}
where
\begin{gather}
\label{qere}
q\e={(d\e)^{n-1} |D|\over h\e \varrho\e(b\e)^n |B| },\ r\e={\varrho\e(b\e)^n |B|\over  \eps^{n-1}}.
\end{gather}

In order to formulate the main results we introduce additional spaces and 
operators.

If $r>0$ then by $\mathcal{H}$ we denote the Hilbert space of functions from $L_2(\Omega)\oplus L_2(\Gamma)$ endowed with the scalar product
\begin{gather*}
(U,V)_{\mathcal{H}}=\intl_{\Omega}u_1(x)\overline{v_1(x)} \d x+\intl_\Gamma u_2(x)\overline{v_2(x)} r \d s,\quad U=\left(u_1,u_2\right),\ V=(v_1,v_2).
\end{gather*}
Hereinafter, we use a standard notation $\d s$ for the density of the measure generated 
on $\Gamma$ (or any other $(n-1)$-dimensional hypersurface) by the Euclidean metrics in $\mathbb{R}^n$.

For $q<\infty$ we introduce the sesquilinear form $\eta_{qr}$ in $\mathcal{H}$ by the formula
\begin{gather*}
\eta_{qr}[U,V]=\intl_\Omega {\nabla u_1}\cdot {\nabla \overline{v_1}}\d x+\intl_\Gamma qr(u_1-u_2)(\overline{v_1-v_2})\d s,\quad U=\left(u_1,u_2\right),\ V=(v_1,v_2)
\end{gather*}
with $\mathrm{dom}(\eta_{qr})=H^1(\Omega)\oplus L_2(\Gamma)$. 
Here we use the same notation for the functions $u_1$, $v_1$ and their traces on $\Gamma$.
This form is densely defined, closed, positive and symmetric. We denote by $\mathcal{A}_{qr}$ the self-adjoint operator associated with this form. Formally, the resolvent equation $\mathcal{A}_{qr} U-\lambda U=F$ (where $U=\left(u_1,u_2\right)$, $F=\left(f_1,f_2\right)$ ) can be written as follows:
\begin{gather*}
\begin{cases}
-\Delta u_1-\lambda u_1=f_1&\text{ in }\Omega,\\ 
{\partial u_1\over\partial  n}+qr(u_1-u_2)=0&\text{ on }\Gamma,\\
q(u_2-u_1)-\lambda u_2=f_2&\text{ on }\Gamma,\\
{\partial u_1\over\partial n}=0&\text{ on }\partial\Omega\setminus\Gamma,
\end{cases}
\end{gather*}
where $n$ is the outward-pointing unit normal to $\Gamma$.

If $q=0$ then $\mathcal{A}_{qr}$  is a direct sum of the operator $\mathcal{A}$ and the null operator in $L_2(\Gamma)$. As a result we have (below by $\sigma_{\mathrm{ess}}$ and $\sigma_{\mathrm{disc}}$ we denote  the essential and the discrete parts of the spectrum):
\begin{gather}\label{q1}
\text{ if }q=0\text{ then }\sigma_{\mathrm{ess}}(\mathcal{A}_{qr})=\{0\},\ \sigma_{\mathrm{disc}}(\mathcal{A}_{qr})=\sigma(\mathcal{A})\setminus\{0\}.
\end{gather}
In the case $q>0$ one has the following result.

\begin{lemma}\label{lm1}
Let $q>0$. Then the spectrum of the operator $\mathcal{A}_{qr}$ has the form
\begin{gather*}
\label{spectrum} \sigma(\mathcal{A}_{qr})=\{q\}\cup 
\{\lambda_k^-,k=1,2,3...\}\cup\{\lambda_k^+,k=1,2,3...\}.
\end{gather*}
The points $\lambda_k^{\pm}, k=1,2,3...$ belong to the discrete
spectrum, $q$ is a point of the essential spectrum and they are distributed  as follows:
\begin{gather}\label{distr}
0= \lambda_1^{+}\leq \lambda_2^{+}\leq ...\leq\lambda_k^{+}\leq\dots\underset{k\to\infty}\to
q<  \lambda_1^{-}\leq
\lambda_2^{-}\leq ...\leq\lambda_k^{-}\leq\dots\underset{k\to\infty}\to \infty.
\end{gather}
\end{lemma}

We present the proof of this lemma \textcolor{black}{at} the end of Subsection 
\ref{subsec22}. At this point we only note that if $\lambda\not= q$ is 
\textcolor{black}{an} eigenvalue of $\mathcal{A}_{qr}$ and $U=(u_1,u_2)$ is the 
corresponding eigenfunction then $u_1\not=0$ and $\lambda$ satisfies 
\eqref{formal1} with $u:=u_1$. Vice versa, if $\lambda\not= q$ and  $u\not=0$ 
satisfies \eqref{formal1} then $\lambda$ is the eigenvalue of 
$\mathcal{A}_{qr}$, 
$U=(u,{qu\over q-\lambda})$ is the corresponding eigenfunction.
 
Also we introduce in $\mathcal{H}$ the sesquilinear form 
\begin{gather*}
\widetilde \eta[U,V]=\intl_\Omega {\nabla u_1}\cdot {\nabla \overline{v_1}}\d x,\ U=(u_1,u_2),\ V=(v_1,v_2)
\end{gather*}
with $\mathrm{dom}(\widetilde\eta)=\left\{U=(u_1,u_2)\in H^1(\Omega)\oplus L_2(\Gamma):\ u_1|_{\Gamma}=u_2\right\}$, where $u_1|_{\Gamma}$ means the trace of $u_1$ on $\Gamma$. The form $\widetilde\eta$ is densely defined, closed, positive and symmetric. We denote by ${\mathcal{A}}_r$ the self-adjoint operator associated with this form. The resolvent equation ${\mathcal{A}}_r U-\lambda U=F$ (where $U=\left(u,u|_\Gamma\right)$, $F=\left(f_1,f_2\right)$) can be written as follows:
\begin{gather*}
\begin{cases}
-\Delta u-\lambda u=f_1&\text{ in }\Omega,\\ 
{\partial u\over\partial  n} -\lambda r u=  r f_2&\text{ on }\Gamma,\\{\partial u\over\partial n}=0&\text{ on }\partial\Omega\setminus \Gamma.
\end{cases}
\end{gather*}
It is clear that ${\mathcal{A}}_r$ has compact resolvent in view of the trace theorem and the Rellich embedding theorem. Therefore the spectrum of ${\mathcal{A}}_r$ is purely discrete.

Finally, for $q<\infty$ we denote by $\mathcal{A}_q$ the operator acting in $L_2(\Omega)\oplus L_2(\Gamma)$ and defined as follows:
$$\mathcal{A}_q=\mathcal{A}\oplus qI,$$
where $I$ is the identity operator in $L_2(\Gamma)$. Obviously, 
\begin{gather}\label{q2}
\sigma_{\mathrm{ess}}(\mathcal{A}_q)=\{q\},\quad 
\sigma_{\mathrm{disc}}(\mathcal{A}_q)=\sigma(\mathcal{A})\setminus\{q\}.
\end{gather}

In what follows speaking about the convergence of spectra we will use the 
concept of the \textit{Hausdorff convergence} (see, e.g., 
\citep{V}). 

{\color{black}
\begin{definition}\label{def1}
The Hausdorff distance between two compact sets $X,Y\subset\mathbb{R}$ 
is defined as follows:
$$\mathrm{dist}_H(X,Y):=\max\limits\left\{\sup\limits_{
x\in X }\inf\limits_{y\in Y} |x-y|;
\sup\limits_{
y\in Y }\inf\limits_{x\in X} |y-x|\right\}.$$
The sequence of sets $X\e\subset\mathbb{R}$ converges to the set 
$X\subset\mathbb{R}$ in the Hausdorff sense if 
$$\mathrm{dist}_H(X\e,X)\to 0\text{ as }\eps\to 0.$$
\end{definition}

Now, we are in position to formulate the main result of this paper. 

\begin{theorem}\label{th1} 
Let $l\subset \mathbb{R}$ be an arbitrary compact interval.
Then the set $\sigma(\mathcal{A}\e)\cap l$ converges in the Hausdorff 
sense as $\eps\to 0$ to the set $\sigma_{0}\cap l$, where
$$\sigma_{0}=\begin{cases}
\sigma(\mathcal{A}_{qr}),&\text{if }q<\infty,\ r>0, \\
\sigma(\mathcal{A}_q),&\text{if }q<\infty,\ r=0, \\
\sigma(\mathcal{A}_r),&\text{if }q=\infty,\ r>0,\\
\sigma(\mathcal{A}),&\text{if }q=\infty,\ r=0.
\end{cases}$$
\end{theorem}

\begin{remark}
It is straightforward to show that the claim of Theorem \ref{th1} is equivalent
to the fulfillment of the following two properties:
\begin{gather}\tag{A}\label{ah}
\text{if }\lambda\e\in\sigma(\mathcal{A}\e)\text{ and }\liml_{\eps\to
0}\lambda\e=\lambda\text{ then }\lambda\in
\sigma_0,\\\tag{B}\label{bh} \text{for any }\lambda\in
\sigma_0\text{ there exists }\lambda\e\in\sigma(\mathcal{A}\e)\text{ such
that }\liml_{\eps\to 0}\lambda\e=\lambda.
\end{gather}
\end{remark}
}

Theorem \ref{th1} will be proved in the next section.
The proof is based on a substitution of suitable
test functions into the variational formulation of the spectral
problem (see equality \eqref{int_eq} below) -- as in the energy method
using for classical homogenization problems (see, e.g., 
\citep{CioDon,Sanchez}).

Before to proceed to the proof of Theorem \ref{th1} we obtain an estimate concerning the 
behaviour of the $k$-th eigenvalue of $\mathcal{A}\e$. 

\begin{theorem}\label{th2}Let $q<\infty$. Then one has
\begin{gather}\label{qqq}
\sup\limits_{k\in\mathbb{N}}\left(\underset{\eps\to 0}{\overline{\lim}}\lambda_k\e\right)\leq q,
\end{gather}
where $\{\lambda_k^\eps\}_{k\in\mathbb{N}}$ is a
sequence of eigenvalues of the operator $\mathcal{A}\e$ written in the ascending order and repeated according to multiplicity.

\end{theorem}

\begin{proof}
By the min-max principle (cf. \citep[\S 4.5]{Davies})
\begin{gather}\label{minmax+}
 \lambda_k\e=\inf\limits_{L\in \mathcal{L}_k}\left(\sup\limits_{0\not= v\in 
L}{\|\nabla v\|^2_{L_2(\Omega\e)}\over \|v\|^2_{\mathcal{H}\e}}\right),
\end{gather}
where $\mathcal{L}_k$ is a set of all $k$-dimensional subspaces in $\mathrm{dom}(\eta\e)$.

Let $\mathbf{i}\e_j\in\I$, $j=1,\dots k$ be arbitrary pairwise non-equivalent indices.
We introduce the following functions:
\begin{gather*}
v_j\e(x)=\begin{cases}
1,& x\in B\e_{\mathbf{i}\e_j},\\
\ds{x_n\over h\e},& x=(x',x_n)\in T\e_{\mathbf{i}\e_j},\\
0,&\text{otherwise}.
\end{cases}
\end{gather*}
We denote 
$$L':=\mathrm{span}\{v_j\e,\ j=1,\dots,k\}.$$
It is clear that $L'\subset H^1(\Omega)$ and $\mathrm{dim}(L')=k$, hence $L'\in\mathcal{L}_k$. Also it is easy to get that
\begin{gather}\label{estv1}
\|\nabla v_j\e\|^2_{L_2(\Omega\e)}={(d\e)^{n-1} |D|\over h\e},
\\\label{estv2}
\| v_j\e\|^2_{\mathcal{H}\e}=\varrho\e (b\e)^n|B|+{1\over 3}(d\e)^{n-1}|D|h\e
\end{gather}
and as a result 
\begin{gather}\label{frac}
{\|\nabla v_j\e\|^2_{L_2(\Omega\e)}\over \|v_j\e\|^2_{\mathcal{H}\e}}
={(d\e)^{n-1} |D|\over h\e \left(\varrho\e(b\e)^n|B|+{1\over 3}(d\e)^{n-1}|D|h\e\right) }=q\e\left (1+{1\over 3}q\e (h\e)^2\right)^{-1}. 
\end{gather}
Since the supports of $v_j\e$ are pairwise disjoint and \eqref{estv1}-\eqref{estv2} are independent of $j$, then, obviously, \eqref{frac} is valid  for an arbitrary $v\in {L}'$ instead of $v_j\e$. Using \eqref{minmax+} and taking into account the finiteness of $q$ and \eqref{ass15}, we obtain
\begin{gather*}
\lambda_k\e\leq \sup\limits_{v\in L'}{\|\nabla v\|^2_{L_2(\Omega\e)}\over \|v\|^2_{\mathcal{H}\e}}=q\e\left (1+{1\over 3}q\e (h\e)^2\right)^{-1}\sim q\text{ as }\eps\to 0, 
\end{gather*}
which implies the statement of the theorem.
\end{proof}

\begin{corollary}
$\text{ If }q=0\text{ then for each }k\in\mathbb{N}\ \lambda_k\e\to 0\text{ as 
}\eps\to 0.$
\end{corollary}

\begin{remark}
In the case $q>0$, $r>0$  \eqref{qqq} follows easily from Theorem \ref{th1} and Lemma \ref{lm1}.
\end{remark}

{\color{black}
\begin{remark}\label{remark_q}
The meaning of the parameter $r$ has been explained in the introduction. 
The parameter $q$ characterizes the ``strength'' of coupling between $\Omega$ 
and the union of ``rooms''. If $q=\infty$ the coupling is strong:
given a family $\left\{w\e\in H^1(\Omega\e)\right\}_{\eps>0}$ satisfying 
\begin{gather}\label{grad}
\|\nabla w\e\|^2_{L_2(\Omega\e)}\leq C
\end{gather}
one can observe that the behaviour as 
$\eps\to 
0$ of $w\e$ on $\cupl_i 
B_i\e$ is 
determined by that one on $\Gamma$. Namely (cf. 
\eqref{u1u2+}),
$$\liml_{\eps\to 0}\left\|\sqrt{r\e}\hspace{1mm} w\e|_\Gamma - \Pi\e 
w\e\right\|_{L_2(\Gamma)}=0,$$
where $w\e|_\Gamma$ is the trace of $w\e$ on $\Gamma$, and 
$\Pi\e:L_2(\cupl_{i}B_i\e)\to L_2(\Gamma)$ is 
a certain bounded 
operator (see Subsection \ref{subsec22}) such that if \eqref{grad} holds then 
$$\|\Pi\e w\e\|^2_{L_2(\Gamma)}\sim  
\suml_{i\in\I}\intl_{B_i\e}|w\e|^2\varrho\e \d x\text{ as }\eps\to 0.$$
On the other hand, if $q$ is finite then the coupling is weak: 
for an arbitrary smooth functions $w_1$, $w_2$ on $\Gamma$ 
one can construct a family 
$\left\{w\e\in H^1(\Omega\e)\right\}_{\eps>0}$ satisfying \eqref{grad} and
$$ 
w\e|_{\Gamma}\to w_1,\ \Pi\e w\e\to w_2\text{ in }L_2(\Gamma)\text{ as }\eps\to 0.$$
\end{remark}}

{\color{black}
\begin{remark}
Condition \eqref{ass2} is rather technical. It is needed to 
control the behaviour of eigenfunctions of the operator $\mathcal{A}\e$ near 
the bottom and the top of the passages -- see estimates \eqref{est1}-\eqref{est2} below. Roughly speaking, eigenfunctions, 
approaching the passage bottom from within $\Omega$ or the passage top from 
within the room, cannot  abruptly ``jump''. 
Apparently, the refusal of this 
condition will not lead to qualitatively new results -- the limit 
spectral problem will still have one of the four forms described by Theorem \ref{th1}, but the coefficient $q$ will be defined by another formula.
\end{remark}
}

\section{\label{sec2}Proof of the main results}

\subsection{\label{subsec21}Preliminaries}

In what follows by $C,C_1,C_2...$ we denote generic constants that do
not depend on $\eps$.

If $G$ is an open domain in $\mathbb{R}^n$ then by $\langle u \rangle_G$ we denote the normalized mean value of the function $u(x)$ in the domain $G$,
$$\langle u \rangle_G={1\over |G|}\intl_G u(x) \d x.$$
If $\Sigma$ is an $(n-1)$-dimensional hypersurface in $\mathbb{R}^n$ then again by $\langle u\rangle_\Sigma$ we denote the normalized mean value of the function $u$ over $\Sigma$, i.e.
$$\langle u\rangle_\Sigma=\ds{1\over |\Sigma|}\intl_{\Sigma}u \d s,\quad |\Sigma|=\intl_{\Sigma}\d s.$$

Next we introduce the following sets for $i\in\I$:
\begin{itemize}

\item $\checkD=\left\{x\in\partial T_i\e:\ x_n=0\right\},$

\item $\hatD=\left\{x\in\partial T_i\e:\ x_n=h\e\right\},$

\item $Y_i\e=\left\{x=(x_1,\dots,x_n)\in\mathbb{R}^n:\ |x_k-(x^{i,\eps})_k|<{\eps\over 2},\ k=1,\dots ,n-1,\ -{\eps\over 2}<x_n<0\right\}$, where $(x^{i,\eps})_k$ is the $k$-th coordinate of $x^{i,\eps}$,

\item $\Gamma_i\e=\left\{x=(x_1,\dots,x_n)\in\mathbb{R}^n:\ |x_k-(x^{i,\eps})_k|<{\eps\over 2},\ k=1,\dots ,n-1,\ x_n=0\right\}.$

\end{itemize}
We have
\begin{gather}\label{disttogamma1}\cupl_{i\in\I}Y_i\e\subset\Omega,\\
\label{disttogamma2}
\cupl_{i\in\I}\Gamma_i\e\subset\Gamma,\quad \liml_{\eps\to 0}\left|\Gamma\setminus \cupl_{i\in\I}\Gamma_i\e\right|=0
\end{gather}
(recall that  the set $\I$ consists of $i\in \mathbb{Z}^{n-1}$  satisfying  $x^{i,\eps}\in\Gamma$ and $\mathrm{dist}(x^{i,\eps},\partial\Omega\setminus\Gamma)\geq \eps {\sqrt{n}\over 2}$, whence one can easily obtain \eqref{disttogamma1}-\eqref{disttogamma2}). 

Further, we present several estimates which will be widely used in the proof of Theorem \ref{th1}.

\begin{lemma}\label{lm2}
One has for $i\in\I$:
\begin{gather}\label{est1}
\forall u\in H^1(Y_i\e):\quad \left|\<u\>_{\checkD}-\<u\>_{\Gamma_{i}^{\eps}}\right|^2\leq C_1 \mathbf{D}\e\|\nabla u\|^2_{L_2(Y_i\e)},\\
\label{est2}\forall u\in H^1(B_i\e):\quad 
\left|\<u\>_{\hatD}-\<u\>_{B_{i}^{\eps}}\right|^2\leq C_2 \mathbf{D}\e\|\nabla u\|^2_{L_2(B_i\e)},
\end{gather}
where $\mathbf{D}\e$ is defined by formula \eqref{De}.
\end{lemma}

\begin{proof} We present the proof of \eqref{est1} only. The proof of \eqref{est2} uses the same ideas and needs only some slight modifications. 

Using density arguments one concludes that it is enough to prove \eqref{est1} only for smooth functions. Let $u\in C^1(\overline{Y_i\e})$.
We denote:
\begin{itemize}
\item $F=\left\{x\in\mathbb{R}^n:\ |x|<{1\over 2},\ x_n<0\right\}$,

\item $D_0=\{x\in\mathbb{R}^n:\ x'\in D,\ x_n=0 \}$,

\item $F_i\e=\left\{x\in\Omega:\ |x-x^{i,\eps}|<{d\e\over 2}\right\}$, 

\item $S_i\e=\left\{x\in\Omega:\ |x-x^{i,\eps}|={d\e\over 2}\right\}$,

\item $R_i\e=\left\{x\in\Omega:\ {d\e\over 2}<|x-x^{i,\eps}|<{\eps\over 2}\right\}$,

\item $C_i\e=\left\{x\in\Omega:\ |x-x^{i,\eps}|={\eps\over 2}\right\}$.
\end{itemize}

In view of \eqref{ass13} one has $D_0\subset\partial F$.
For an arbitrary function $v\in H^1(F)$ one has the standard trace inequality:
\begin{gather}\label{trace}
\|v\|^2_{L_2(D_0)}\leq C\|v\|^2_{H^1(F)}.
\end{gather}
Then via the change of variables $x\mapsto d\e x+x_i\e$ (mapping $D_0$ onto $\checkD$ and $F$ onto $F_i\e$) one can easily obtain from \eqref{trace}:
\begin{gather}\label{tr1}
\forall u\in H^1(F_i\e):\ \|u\|^2_{L_2(\checkD)}\leq C\left((d\e)^{-1}\|u\|^2_{L_2(F_i\e)}+
d\e\|\nabla u\|^2_{L_2(F_i\e)}\right).
\end{gather}
In a similar way we also get
\begin{gather}
\label{tr2}
\forall u\in H^1(F_i\e):\ \|u\|^2_{L_2(S_i\e)}\leq C\left((d\e)^{-1}\|u\|^2_{L_2(F_i\e)}+
d\e\|\nabla u\|^2_{L_2(F_i\e)}\right).
\end{gather}
Then, using the Cauchy inequality,  \eqref{tr1}-\eqref{tr2} and the Poincar\'{e} inequality 
$$\|u-\langle u\rangle_{F_i\e}\|^2_{L_2(F_i\e)}\leq C (d\e)^2 \|\nabla u\|^2_{L_2(F_i\e)}$$
we obtain:
\begin{multline}\label{tr3}
\left|\langle u\rangle_{\checkD}-\langle u\rangle_{S_i\e}\right|^2\leq 2
\left|\langle u\rangle_{\checkD}-\langle u\rangle_{F_i\e}\right|^2+
2\left|\langle u\rangle_{S_i\e}-\langle u\rangle_{F_i\e}\right|^2\leq 
{2\over |\checkD|}\|u-\langle u\rangle_{F_i\e}\|^2_{L_2(\checkD)}+
{2\over |S_i\e|}\|u-\langle u\rangle_{F_i\e}\|^2_{L_2(S_i\e)}\\\leq
C (d\e)^{1-n}\left((d\e)^{-1}\|u-\langle u\rangle_{F_i\e}\|^2_{L_2(F_i\e)}+d\e\|\nabla u\|^2_{L_2(F_i\e)}\right)\leq C_1(d\e)^{2-n}\|\nabla u\|^2_{L_2(F_i\e)}.
\end{multline}

By $\Sigma_{n-1}$ we denote $(n-1)$-dimensional unit half-sphere and 
introduce spherical coordinates $(\phi,r)$ in $R_i\e$. Here
$\phi=(\phi_1,{\dots},\phi_{n-1})\in \Sigma_{n-1}$ are the angular coordinates, $r\in \left({d\e\over 2},{\eps\over 2}\right)$ is a distance to $x^{i,\eps}$.

Let $x=(\phi,d\e/2)\in S_i\e$, $y=(\phi,\eps/2)\in C_i\e$. We have
\begin{gather*}
u(y)-u(x)=\intl_{0}^{{\eps\over 2}-{d\e\over 2}}{\partial
u(\xi(\tau))\over\partial\tau}\d\tau,\text{ where }\xi(\tau)=x+{\tau\over {\eps\over 2}-{d\e\over 2}}(x-y).
\end{gather*}
Then we integrate this equality over
$\Sigma_{n-1}$ with respect to $\phi$, divide by
$|\Sigma_{n-1}|$ and square. Using the Cauchy inequality we obtain:
\begin{multline}\label{tr4}
\left|\langle u\rangle_{C_i\e}-\langle
u\rangle_{S_i\e}\right|^2=
{1\over|\Sigma_{n-1}|^2}\left|\hspace{2pt}
\intl_{\Sigma_{n-1}}\intl_{0}^{{\eps\over 2}-{d\e\over 2}}{\partial
u(\xi(\tau))\over\partial\tau}\d\tau \d\phi
\right|^2\leq
C
\left\{\intl_{0}^{{\eps\over 2}-{d\e\over 2}}\left(\tau+{d\e\over 2}\right)^{1-n}{\d\tau}\right\}\\\times\left\{\intl_{\Sigma_{n-1}}\intl_{0}^{{\eps\over 2}-{d\e\over 2}}\left|{\partial
u(\xi(\tau))\over\partial\tau}\right|^2\left(\tau+{d\e\over 2}\right)^{n-1}\d\tau
\d\phi\right\}\leq C_1 \mathbf{D}\e {\|\nabla
u\|^2_{L_2(R_i\e)}}.
\end{multline}

Finally, using the same idea as in the proof of \eqref{tr3}, we obtain the estimate
\begin{gather}\label{tr5}
\left|\langle u\rangle_{C_i\e}-\langle u\rangle_{\Gamma_i\e}\right|^2\leq C\eps^{2-n}\|\nabla u\|^2_{L_2(Y_i^{\eps})}.
\end{gather}
Combining \eqref{tr3}-\eqref{tr5}  and taking into account that $(d\e)^{2-n}+\eps^{2-n}\leq 2\mathbf{D}\e$ we obtain \eqref{est1}. 
\end{proof}
\begin{lemma}\label{lm3}
One has  for $i\in \I$:
\begin{gather}\label{est3}
\forall u\in H^1(T_i\e):\quad \left|\<u\>_{\checkD}-\<u\>_{\hatD}\right|^2\leq {C\over q\e r\e\eps^{n-1}}\|\nabla u\|^2_{L_2(T_i\e)},
\end{gather}
where $q\e$ and $r\e$ are defined by \eqref{qere}.
\end{lemma}

\begin{proof}
It is enough to prove \eqref{est3} only for smooth functions.
Let $u$ be an arbitrary function from $C^1(\overline{T_i\e})$.
Let $x=(x',0)\in \checkD$, $y=(x',h\e)\in \hatD$.
One has
\begin{gather*}
u(y)-u(x)=\intl_{0}^{h\e}{\partial
u(\xi(\tau))\over\partial\tau}\d\tau,\text{ where }\xi(\tau)=x+{\tau\over {h\e}}(x-y).
\end{gather*}
We integrate this equality over $D_i\e:=d\e D+i\eps$ with respect to $x'$ , then divide by
$|D_i\e|$ and square. Using Cauchy inequality we obtain:
\begin{multline*}
\left|\langle u\rangle_{\hatD}-\langle u\rangle_{\checkD}\right|^2=
{1\over |D_i\e|^2}\left|\hspace{2pt}
\intl_{D_i\e}\intl_{0}^{h\e}{\partial
u(\xi(\tau))\over\partial\tau} \d\tau \d x'
\right|^2\leq
C{h\e\over (d\e)^{n-1}} \|\nabla u\|^2_{L_2(T_i\e)}={C_1\over q\e r\e \eps^{n-1}}\|\nabla u\|^2_{L_2(T_i\e)} 
\end{multline*}
and \eqref{est3} is proved.
\end{proof}

\begin{lemma}\label{lm4}
One has
\begin{gather}\label{est4}
\forall u\in H^1(\Omega\e):\quad \suml_{i\in\I}\|u\|^2_{L_2(T_i\e)}\leq C h\e\left(\|u\|^2_{H^1(\Omega)}+\suml_{i\in \I}\|\nabla u\|_{L_2(T_i\e)}^2\right).
\end{gather}
\end{lemma}

\begin{proof}
It is enough to prove the lemma only for smooth functions. Let $u$ be an arbitrary function from $C^1(\overline{\Omega\e})$. For the sake of simplicity we suppose that there exists $a>0$ such that the set
$$\Omega_{a}:=\left\{x=(x',x_n)\in\mathbb{R}^n:\ (x',0)\in\Gamma,\ x_n\in (-a,0)\right\}$$
is a subset of $\Omega$.  For the general case the proof need some small modifications.

Let $x=(x',\textbf{x})$, $y=(x',\textbf{y})$, where $x'\in D_i\e:=d\e D+i\eps$, $\mathbf{x}\in (-a,0)$, $\mathbf{y}\in (0,h\e)$.
One has
\begin{gather*}
u(y)=u(x)+\intl_{0}^{\textbf{y}-\textbf{x}}{\partial
u(\xi(\tau))\over\partial\tau}\d\tau,\text{ where }\xi(\tau)=x+{\tau\over y-x}\left( \mathbf{y}-\mathbf{x}\right).
\end{gather*}
We square this equality, then integrate over $D_i\e$ with respect to $x'$, over $(-a,0)$ with respect to $\textbf{x}$ and over $(0,h\e)$ with respect to $\textbf{y}$. We arrive at
\begin{multline}\label{tr8}
a\|u\|^2_{L_2(T_i\e)}=\intl_0^{h\e}\intl_{-a}^0\intl_{D_i\e}
\left|u(x)+\intl_{0}^{\mathbf{y}-\mathbf{x}}{\partial
u(\xi(\tau))\over\partial\tau} \d\tau\right|^2 \d x' \d\mathbf{x}\d\mathbf{y}\\\leq
2h\e\|u\|^2_{L_2(\widetilde T_i\e)}+2h\e a (a+h\e)\|\nabla u\|_{L_2(\widetilde T_i\e\cup T_i\e)}^2,
\end{multline}
where $\widetilde T_i\e=\{x=(x',x_n):\ x'\in D_i\e,\ -a<x_n<0\}$.
Summing up \eqref{tr8} by $i\in\I$ and taking into account, that $\cupl_{i\in\I}\widetilde T_i\e\subset\Omega_a\subset\Omega$,
we obtain the required inequality \eqref{est4}. 
\end{proof}

\subsection{\label{subsec22}Proof of Theorem \ref{th1}: the case $q<\infty $}

\subsubsection{Proof of the property (A) of Hausdorff convergence} Let $\lambda\e\in\sigma(\mathcal{A}\e)$ and  $\lambda\e\to\lambda$ as $\eps\to 0$. We have to prove that either 
\begin{gather}\label{HausA1}
\lambda\in\sigma(\mathcal{A}_{qr})\text{ if }r>0\text{\quad or\quad  }\lambda\in\sigma(\mathcal{A}_q)\text{ if }r=0.
\end{gather}

Recall, that by $\{\lambda_k^\eps\}_{k=1}^\infty$ we denote the
sequence of eigenvalues of $\mathcal{A}\e$ written in the ascending order and repeated according to multiplicity. By $\{u_k\}_{k=1}^\infty$ we denote a corresponding sequence of eigenfunctions normalized by the condition $(u_k\e,u_l\e)_{\mathcal{H}\e}=\delta_{kl}$.

We denote by $k\e$ the index corresponding to $\lambda\e$ (i.e. $\lambda\e=\lambda_{k\e}\e$).
By $u\e=u\e_{k\e}\in H^1(\Omega)$ we denote the corresponding eigenfunction. One has
\begin{gather}\label{normalization}
\|u\e\|_{\mathcal{H}\e}=1,\quad \|\nabla u\e\|^2_{L_2(\Omega\e)}=\lambda\e.
\end{gather}

In order to describe the behaviour of $u\e$ on $\cupl_{i\in \I} B_i\e$ as $\eps\to 0$ we will use the operator $$\Pi\e:L_2(\cupl_{i\in\I} B_i\e)\to L_2(\Gamma)$$ defined as follows:
\begin{gather*}
\Pi\e u(x)=
\begin{cases}
\langle u \rangle_{B_i\e}\sqrt{r\e},& x\in \Gamma^{\eps}_i,\\ 0,& x\in\Gamma\backslash\cupl_{i\in\I}
\Gamma^{\eps}_i.
\end{cases}
\end{gather*}
(recall, that $r\e$ is defined by \eqref{qere}). 
Using the Cauchy inequality, \eqref{disttogamma2} and taking into account that 
$$|\Gamma_i\e|=\eps^{n-1},\quad |B_i\e|=(b\e)^n |B|$$
we obtain
\begin{gather}
\label{Pi_ineq} \|\Pi\e u\|^2_{L_2(\Gamma)}\leq r\e\suml_{i\in\mathcal{I}\e}{|\Gamma_i\e|\over|B_i\e|}
\intl_{B_i\e}|u(x)|^2\d x=\suml_{i\in\mathcal{I}\e}\intl_{B_i\e}\varrho\e|u(x)|^2\d x\leq \|u\|^2_{\mathcal{H}\e}.
\end{gather}

In view of \eqref{normalization}, \eqref{Pi_ineq}
$$\|u\e\|^2_{H^1(\Omega)}+\|\Pi\e u\e\|^2_{L_2(\Gamma)}\leq C,$$
whence, using the Rellich embedding theorem and the trace theorem, we conclude that there is a subsequence (still denoted by $\eps$) and $u_1\in H^1(\Omega)$, $u_2\in L_2(\Gamma)$ such that
\begin{gather}\label{conv1}
u\e\rightharpoonup u_1\text{ in }H^1(\Omega),\\\label{conv2}
u\e\rightarrow u_1\text{ in }L_2(\Omega),\\\label{conv3}
u\e\rightarrow u_1\text{ in }L_2(\Gamma),\\
\label{conv4}
\Pi\e u\e\rightharpoonup u_2\text{ in }L_2(\Gamma)
\end{gather}
as $\eps\to 0$ (here we use the same notation for the functions $u\e$, $u_1$ and their traces on $\Gamma$).

We start from the case 
\begin{gather*}
u_1\not=0. 
\end{gather*}
We will prove that $\lambda$ is the eigenvalue of the operator 
$\textcolor{black}{\mathcal{A}_{qr}}$ if $r>0$ (respectively, of the operator 
$\mathcal{A}_{q}$ if $r=0$)
and $U=(u_1,r^{-1/2}u_2)$ (respectively, $U=(u_1,u_2)$) is the corresponding eigenfunction. 

For an arbitrary $w\in H^1(\Omega\e)$  we have
\begin{gather}\label{int_eq}
\intl_{\Omega\e}\nabla u\e\cdot\nabla w \d x=\lambda\e\intl_{\Omega\e} u\e w  \rho\e \d x.
\end{gather}
The strategy of proof is to plug into \eqref{int_eq} some specially chosen test-function $w$ depending on $\eps$ and then pass to the limit as $\eps\to 0$ in order to obtain either the equality $\mathcal{A}_{qr} U=\lambda U$ ($r>0$) or the equality $\mathcal{A}_q U=\lambda U$ ($r=0$)  written in a weak form. 

We choose this test-function as follows:
\begin{gather}\label{we}
w(x)=w\e(x):=
\begin{cases}
w_1(x)+\ds\suml_{i\in \I}(w_1(x^{i,\eps})-w_1(x))\varphi_i\e(x),& x\in\Omega,\\
\ds{{1\over\sqrt{r\e}}w_2(x^{i,\eps})-w_1(x^{i,\eps})\over h\e}x_n+w_1(x^{i,\eps}),& x=(x',x_n)\in T_i\e,\\
\ds{1\over\sqrt{r\e}}w_2(x^{i,\eps}),&x\in B_i\e.
\end{cases}
\end{gather}
Here $w_1\in C^\infty(\Omega)$, $w_2\in C^\infty(\Gamma)$ are arbitrary functions, $\varphi_i\e(x)=\varphi\left({|x-x^{i,\eps}|\over\eps}\right)$, where $\varphi:\mathbb{R}\to\mathbb{R}$ is a smooth functions satisfying $\varphi(t)=1$ as $t\leq R$ and $\varphi(t)=0$ as $t\geq {1\over 2}$, the constant $R\in (0,{1\over 2})$ comes from \eqref{ass12}-\eqref{ass13}. It is clear that $w\e(x)$ is continuous and piecewise smooth function.

We plug $w=w\e(x)$  into \eqref{int_eq}. Firstly, we study the left-hand-side. 
Taking into account that $\supp(\varphi_i\e)\subset\overline{Y_i\e}$ and $w\e=\mathrm{const}$ in $B_i\e$
we obtain:
\begin{multline}\label{Iomega}
\intl_{\Omega\e}\nabla u\e \cdot\nabla w\e\d x=\intl_{\Omega}\nabla u\e\cdot \nabla w_1 \d x+\suml_{i\in\I}\intl_{Y_i\e}\nabla u\e\cdot\nabla \left(\left(w_1(x^{i,\eps})-w_1\right)\phi_i\e\right)\d x
\\+\suml_{i\in\I}\intl_{T_i\e}\nabla u\e\cdot\nabla w\e \d x.
\end{multline}

By virtue of \eqref{conv1} we get
\begin{gather}\label{Iomega1}
\intl_{\Omega}\nabla u\e\cdot \nabla w_1 \d x\to \intl_{\Omega}\nabla u_1\cdot\nabla w_1 \d x\text{ as }\eps\to 0.
\end{gather}

Using $\supp(\varphi_i\e)\subset\overline{Y_i\e}$ one can easily obtain the estimate
\begin{gather*}
\left|\nabla \left(\left(w_1(x^{i,\eps})-w_1\right)\phi_i\e\right)\right|\leq C,
\end{gather*}
whence, taking into account that \begin{gather}\label{sumeps}\suml_{i\in\mathcal{I}\e}\eps^{n-1}=\suml_{i\in\mathcal{I}\e}|\Gamma_i\e|\leq |\Gamma|,
\end{gather}
we obtain:
\begin{gather}\label{Iomega2}
\left|\suml_{i\in\I}\intl_{Y_i\e}\nabla u\e\cdot\nabla \left(\left(w_1(x^{i,\eps})-w_1\right)\phi_i\e\right)\d x\right|^2\leq C\|\nabla u\e\|^2_{L_2(\Omega)}\left|\cupl_{i\in\mathcal{I}\e} Y_i\e\right|\leq C_1\eps\suml_{i\in\mathcal{I}\e}\eps^{n-1}\underset{\eps\to 0}\to 0.
\end{gather}

Now we inspect the third integral in \eqref{Iomega}. Integrating by parts and taking into account that $\Delta w\e=0$ in $T_i\e$  we get:
\begin{multline}\label{Iomega3}
\suml_{i\in\I}\intl_{T\e}\nabla u\e\cdot\nabla w\e \d x=
-\intl_{\checkD}u\e {\partial w\e\over\partial x_n} \d s+\intl_{\hatD}u\e {\partial w\e\over\partial x_n} \d s\\=
\suml_{i\in\I}{(d\e)^{n-1}|D|\over h\e}\left(\langle u\e\rangle_{\checkD} - \langle u\e\rangle_{\hatD}\right)\left(w_1(x^{i,\eps})-{w_2(x^{i,\eps})\over\sqrt{r\e}}\right)\\=
q\e r\e\suml_{i\in\I}{\eps^{n-1}}\left(\langle u\e\rangle_{\Gamma_i^{\eps}} - \langle u\e\rangle_{B_i^{\eps}}\right)\left(w_1(x^{i,\eps})-{w_2(x^{i,\eps})\over\sqrt{r\e}}\right)+\delta(\eps),
\end{multline}
where $q\e$ and $r\e$ are defined by \eqref{qere} and the remainder $\delta(\eps)$ vanishes as $\eps\to 0$, namely, using the Cauchy inequality, condition \eqref{ass2}, estimates \eqref{est1}, \eqref{est2} and \eqref{sumeps}, we obtain:
\begin{multline}\label{delta_est}
|\delta(\eps)|^2\leq C(q\e r\e)^2\suml_{i\in\I}\eps^{n-1}\left( \left|\langle u\e\rangle_{\checkD} - \langle u\e\rangle_{\Gamma_i^{\eps}}\right|^2+ \left|\langle u\e\rangle_{\hatD} - \langle u\e\rangle_{B_i^{\eps}}\right|^2\right)\cdot \suml_{i\in\I}\eps^{n-1}{1\over r\e}\\\leq C_1(q\e)^2r\e\mathbf{D}\e\eps^{n-1}\|\nabla u\e\|_{L_2(\cupl_{i\in \I}(Y_i\e\cup B_i\e))}^2\leq 
C_2\mathbf{D}\e\eps^{n-1}\to 0\text{ as }\eps\to 0.
\end{multline}

We introduce the operator $Q\e: C^1(\Gamma)\to L_2(\Gamma)$ defined by the formula
\begin{gather*}
Q\e w=\begin{cases}
w(x^{i,\eps}),& x\in \Gamma^{\eps}_i,\\ 0,& x\in\Gamma\backslash\cupl_{i\in\I}
\Gamma^{\eps}_i.
\end{cases}
\end{gather*}
It is straightforward to show that
\begin{gather}\label{Q}
\forall w\in C^1(\Gamma),\ Q\e w\underset{\eps\to 0}\to w\text{ in }L_2(\Gamma).
\end{gather}
Then, taking into account the definitions of the operators $\Pi\e$, $Q\e$ and 
using \eqref{qere+}, \eqref{conv3}, \eqref{conv4}, \eqref{delta_est}, \eqref{Q}, we obtain from \eqref{Iomega3}:
\begin{multline}\label{Iomega4}
\suml_{i\in\I}\intl_{T\e}\nabla u\e\cdot\nabla w\e \d x=q\e r\e\intl_{\Gamma}\left(u\e-{1\over \sqrt{r\e}}\Pi\e u\e\right)\left(Q\e w_1-{1\over\sqrt{r\e}}Q\e w_2\right)\d s+\delta(\eps)\\\to \intl_{\Gamma}\left(qr u_1 w_1-q\sqrt{r}u_2 w_1-q\sqrt{r}u_1 w_2+q u_2 w_2\right)\d s\text{ as }\eps\to 0.
\end{multline}

Combining  \eqref{Iomega}-\eqref{Iomega2}, \eqref{Iomega4} we arrive at
\begin{gather}\label{LHS}
\intl_{\Omega}\nabla u\e\cdot\nabla w\e \d x\underset{\eps\to 0}\to \intl_{\Omega} \nabla u_1\cdot\nabla w_1 \d x+\intl_{\Gamma}\left(qr u_1 w_1-q\sqrt{r}u_2 w_1-q\sqrt{r}u_1 w_2+q u_2 w_2\right)\d s.
\end{gather}

Now, we study the right-hand-side of \eqref{int_eq}. One has:
\begin{multline}\label{Jomega}
\lambda\e\intl_{\Omega\e} u\e w\e \rho\e \d x=\lambda\e\left(\intl_{\Omega}u\e w_1 \d x+\suml_{i\in\I}\intl_{Y_i\e} u\e \left(w_1(x^{i,\eps})-w_1\right)\phi_i\e \d x\right.\\\left.+\suml_{i\in\I}\intl_{T_i\e}u\e w\e \d x+{\varrho\e\over\sqrt{r\e}}\suml_{i\in\I}\intl_{B_i\e} u\e w_2(x^{i,\eps})\d x\right).
\end{multline}

It is clear that 
\begin{gather}\label{Jomega1}
\intl_{\Omega}u\e w_1  \d x\to \intl_{\Omega}u_1 w_1 \d x\text{ as }\eps\to 0
\end{gather}
and the next two integrals in \eqref{Jomega} vanishes as $\eps\to 0$:
\begin{gather}\label{Jomega2}
\left|\suml_{i\in\I}\intl_{Y_i\e} u\e \left(w_1(x^{i,\eps})-w_1\right)\phi_i\e \d x\right|^2\leq 
\suml_{i\in\I}\|w_1(x^{i,\eps})-w_1\|^2_{L_2(Y_i\e)}\suml_{i\in \I}\|u\e\|^2_{L_2(Y_i\e)}\leq 
C\suml_{i\in\I}\eps^{n+2}\to 0,\\\label{Jomega2+}
\left|\suml_{i\in\I}\intl_{T_i\e}u\e w\e \d x\right|^2\leq 
\suml_{i\in \I}\|u\e\|^2_{L_2(T_i\e)} \suml_{i\in\I}\|w\e\|^2_{L_2(T_i\e)}\leq
C\suml_{i\in\I}{(d\e)^{n-1} h\e\over{{r\e}}}\leq C_1 q\e (h\e)^2\suml_{i\in \I}\eps^{n-1} \to 0.
\end{gather}

Finally, we inspect the behaviour of the last integral in \eqref{Jomega}. One has:
\begin{multline}\label{Jomega3}
\suml_{i\in\I}\intl_{B_i\e}{\varrho\e\over\sqrt{r\e}} u\e w_2(x^{i,\eps})\d x=\suml_{i\in\I}{\varrho\e (b\e)^n|B|\over\sqrt{r\e}}\langle u\e\rangle_{B_i\e}w_2(x^{i,\eps})=\intl_\Gamma \Pi\e u\e\hspace{1pt} Q\e w_2 \d s\underset{\eps\to 0}\to \intl_{\Gamma}u_2 w_2 \d s.
\end{multline}

It follows from \eqref{Jomega}-\eqref{Jomega3} and $\liml_{\eps\to 0}\lambda\e=\lambda$ that
\begin{gather}\label{RHS}
\liml_{\eps\to 0}\left(\lambda\e\intl_{\Omega\e} u\e w\e \rho\e \d x\right)=\lambda\left(\intl_{\Omega}u_1 w_1 \d x+\intl_{\Gamma}u_2 w_2 \d s\right).
\end{gather}

Finally, combining \eqref{int_eq}, \eqref{LHS} and \eqref{RHS}, we get
\begin{gather}\label{int_eq_final1}
\intl_{\Omega}\nabla u_1\cdot \nabla w_1 \d x+\intl_{\Gamma}\left(qr u_1 w_1-q\sqrt{r}u_2 w_1-q\sqrt{r}u_1 w_2+q u_2 w_2\right)\d s=\lambda\left(\intl_{\Omega}u_1 w_1 \d x+\intl_{\Gamma}u_2 w_2 \d s\right).
\end{gather}
By the density arguments equality \eqref{int_eq_final1} is valid for an arbitrary $(w_1,w_2)\in H^1(\Omega)\oplus L_2(\Gamma)$.

If $r>0$ then \eqref{int_eq_final1} is equivalent to equality
\begin{gather*}
\eta_{qr}[U,W]=\lambda(U,W)_{\mathcal{H}},\text{ where }U=(u_1,r^{-1/2}u_2),\ W=(w_1,r^{-1/2}w_2),
\end{gather*}
whence, evidently, 
\begin{gather*}
U\in\mathrm{dom}(\mathcal{A}_{qr}),\ \mathcal{A}_{qr}U=\lambda U,
\end{gather*}
and therefore, since $u_1\not=0$, $\lambda$ is the eigenvalue of the operator $\mathcal{A}_{qr}$. 
If $r=0$ then \eqref{int_eq_final1} implies
\begin{gather*}
U=(u_1,u_2)\in\mathrm{dom}(\mathcal{A}_{q}),\ \mathcal{A}_{q}U=\lambda U, 
\end{gather*}
i.e. $\lambda$ is the eigenvalue of the operator $\mathcal{A}_{q}$.
\medskip

Now, we inspect the case $$u_1=0.$$
We will prove that in this instance 
$\lambda=q$. 
Recall (see \eqref{q1}, Lemma \ref{lm1} and \eqref{q2}) that the point $q$ belongs to the essential spectrum of both  operators $\mathcal{A}_{qr}$ and $\mathcal{A}_{q}$.

We express the eigenfunction $u\e$ in the form
\begin{gather}\label{repres}
u\e=v\e-g\e+w\e,
\end{gather}
where 
\begin{gather*}
v\e(x)=
\begin{cases}
0,&x\in \Omega,\\
{\<u\e\>_{B_i\e}\over h\e}x_n,&x=(x',x_n)\in T_i\e,\\
\<u\e\>_{B_i\e},&x\in B_i\e
\end{cases}
\end{gather*}
and 
$$g\e=\suml_{i=1}^{k\e-1}(v\e,u_k\e)_{\mathcal{H}\e} u_k\e.$$
(recall, that  $\lambda\e=\lambda\e_{k\e}$, $u\e=u\e_{k\e}$). It is clear that $v\e\in H^1(\Omega\e)$, $g\e\in\mathrm{dom}(\mathcal{A}\e)$ and 
\begin{gather}\label{g_prop1}
v\e-g\e\in \left(\mathrm{span}\left\{u_1\e,\dots,u_{k\e-1}\e\right\}\right)^\perp. 
\end{gather}

One has the following Poincar\'{e}-type inequality:
\begin{gather}\label{prelim1}
\suml_{i\in\I}\intl_{B_i\e}\varrho\e\left|u\e-\<u\e\>_{B_i\e}\right|^2\d x \leq C\varrho\e (b\e)^2\suml_{i\in\I}\|\nabla u\e\|_{L_2(B_i\e)}^2\leq C_1 \varrho\e (b\e)^2.
\end{gather}
Since $b\e\geq d\e$ and $r\e\leq C$ then
\begin{gather*}
\varrho\e (b\e)^2 = C r\e {\eps^{n-1}\over (b\e)^{n-2}} \leq C r\e {\eps^{n-1}\over (d\e)^{n-2}}=C_1\cdot\begin{cases} \mathbf{D}\e\eps^{n-1},&n\geq 2,\\\eps,&n=2,\end{cases}
\end{gather*}
and therefore in view of \eqref{ass2} 
\begin{gather}\label{poincare}
\varrho\e  (b\e)^2 \to 0\text{ as }\eps\to 0.
\end{gather}
It follows from \eqref{prelim1}, \eqref{poincare} that
\begin{gather}\label{uB}
\suml_{i\in\I}\intl_{B_i\e}\varrho\e\left|u\e-\<u\e\>_{B_i\e}\right|^2\d x \to 0\text{ as }\eps\to 0.
\end{gather}

Using estimate \eqref{est4} we get
\begin{gather}\label{uG}
\|u\e\|_{L_2(\cupl_{i\in\I} T_i\e)}\to 0\text{ as }\eps\to 0.
\end{gather}

Taking into account $\|u\e\|_{\mathcal{H}\e}=1$ one can easily show that
\begin{gather*}
\suml_{i\in\I}r\e\eps^{n-1} \left|\<u\e\>_{B_i\e}\right|^2=1-\|u\e\|^2_{L_2(\Omega)}-\suml_{i\in\I}\|u\e\|_{L_2(T_i\e)}^2- \suml_{i\in\I}\intl_{B_i\e}\varrho\e\left
|u\e-\<u\e\>_{B_i\e}\right|^2\d x,
\end{gather*}
whence, in view of \eqref{uB}, \eqref{uG} and the fact that $u_1=0$, we obtain
\begin{gather}\label{1}
\suml_{i\in\I}r\e\eps^{n-1} \left|\<u\e\>_{B_i\e}\right|^2=1+o(1)\text{ as }\eps\to 0.
\end{gather}

Using \eqref{1} one has the following asymptotics for the function $v\e$:
\begin{gather}\label{v1}
\|\nabla v\e\|^2_{L_2(\Omega\e)}=
q\e\suml_{i\in\I}r\e\eps^{n-1}\left|\<u\e\>_{B_i\e}\right|^2=q + o(1)\text{ as }\eps\to 0,\\\label{v2}
\suml_{i\in\I}\intl_{B_i\e}\varrho\e |v\e|^2 \d x= \suml_{i\in\I}r\e\eps^{n-1}\left| \<u\e\>_{B_i\e}\right|^2 =1+o(1)\text{ as }\eps\to 0,\\\label{v3}
\suml_{i\in\I}\|v\e\|^2_{L_2(T_i\e)}={1\over 3}q\e (h\e)^2\suml_{i\in\I}r\e\eps^{n-1}\left| \<u\e\>_{B_i\e}\right|^2
=o(1)\text{ as }\eps\to 0.
\end{gather}
It follows from \eqref{v2}-\eqref{v3} and $v\e=0$ in $\Omega$ that
\begin{gather}\label{v4}
\|v\e\|_{\mathcal{H}\e}=1+o(1)\text{ as }\eps\to 0.
\end{gather}

By virtue of \eqref{uB}, \eqref{uG}, \eqref{v3} and the fact that $\|u\e\|^2_{L_2(\Omega)}\underset{\eps\to 0}\to \|u_1\|^2_{L_2(\Omega)}= 0$ one gets:
\begin{gather}\label{u-v}
\|u\e-v\e\|^2_{\mathcal{H}\e}\leq \suml_{i\in\I}\intl_{B_i\e}\varrho\e\left|u\e-\<u\e\>_{B_i\e}\right|^2 \d x+2\|u\e\|^2_{L_2(\cupl_{i\in\I} T_i\e)}+2\|v\e\|^2_{L_2(\cupl_{i\in\I} T_i\e)}+\|u\e\|^2_{L_2(\Omega)}\underset{\eps\to 0}\to 0.
\end{gather}

Using the equality $(u\e,u_k\e)_{\mathcal{H}\e}=0$ for $k=1,\dots, k\e-1$ and the Bessel inequality we obtain:
\begin{gather*}
\|g\e\|^2_{\mathcal{H}\e}=\suml_{k=1}^{k\e-1}\left|(v\e,u_k\e)_{\mathcal{H}\e}\right|^2=
\suml_{k=1}^{k\e-1}\left|(v\e-u\e,u_k\e)_{\mathcal{H}\e}\right|^2\leq \|v\e-u\e\|^2_{\mathcal{H}\e},\\
\|\nabla g\e\|^2_{L_2(\Omega\e)}=\suml_{k=1}^{k\e-1}\lambda_k\e\left|(v\e, u_k\e)_{\mathcal{H}\e}\right|^2=
\suml_{k=1}^{k\e-1}\lambda_k\e\left|(v\e-u\e,u_k\e)_{\mathcal{H}\e}\right|^2\leq \lambda\e\|v\e-u\e\|^2_{\mathcal{H}\e}
\end{gather*}
and thus in view of \eqref{u-v} 
\begin{gather}\label{g_est}
\|g\e\|^2_{\mathcal{H}\e}+\|\nabla g\e\|^2_{L_2(\Omega\e)}\to 0\text{ as }\eps\to 0.
\end{gather}

Now let us estimate the remainder $w\e$. 
It is well-known (cf. \citep{Reed}) that
\begin{gather}\label{minmax_p}
\lambda\e=\inf\left\{{\|\nabla u\|^2_{L_2(\Omega\e)}\over \| u\|^2_{\mathcal{H}\e}},\ 0\not=u\in \left(\mathrm{span}\left\{u_1\e,\dots,u_{k\e-1}\e\right\}\right)^\perp \right\}.
\end{gather}

We denote $\tilde v\e=v\e-g\e$. Taking into account \eqref{normalization} and \eqref{g_prop1} we obtain from \eqref{minmax_p}:
\begin{gather*}
\|\nabla u\e\|^2_{L_2(\Omega\e)}\leq {\|\nabla \tilde v\e\|^2_{L_2(\Omega\e)}\over \|\tilde v\e\|^2_{\mathcal{H}\e}}
\end{gather*}
or, using $u\e=\tilde v\e+w\e$,
\begin{gather}\label{minmax}
\|\nabla w\e\|^2_{L_2(\Omega\e)}\leq -2(\nabla \tilde v\e,\nabla w\e)_{L_2(\Omega\e)}+{\|\nabla \tilde v\e\|^2_{L_2(\Omega\e)}}\left(\|\tilde v\e\|^{-2}_{\mathcal{H}\e}-1\right).
\end{gather}

In view of \eqref{v1}, \eqref{v4}, \eqref{g_est}
\begin{gather}\label{minmax1}
{\|\nabla \tilde v\e\|^2_{L_2(\Omega\e)}}\left(\|\tilde v\e\|^{-2}_{\mathcal{H}\e}-1\right)\to 0\text{ as }\eps\to 0.
\end{gather}

Let us estimate the first term in the right-hand-side of \eqref{minmax}. One has
\begin{gather}\label{minmax2}
(\nabla \tilde v\e,\nabla w\e)_{L_2(\Omega\e)}=(\nabla v\e,\nabla u\e-\nabla v\e)_{L_2(\Omega\e)}+(\nabla v\e,\nabla g\e)_{L_2(\Omega\e)}-(\nabla g\e,\nabla w\e)_{L_2(\Omega\e)}.
\end{gather}
Integrating by parts we obtain:
\begin{multline}\label{delta1}
(\nabla v\e,\nabla u\e-\nabla v\e)_{L_2(\Omega\e)}=\suml_{i\in \I}\intl_{T_i\e} \nabla v\e\cdot\nabla (u\e-v\e)\d x=
\suml_{i\in \I} \left(\intl_{\checkD}-{\partial v\e\over\partial x_n}u\e \d s+\intl_{\hatD}{\partial v\e\over\partial x_n}(u\e-\langle u\e \rangle_{B_i\e}) \d s\right)\\
={(d\e)^{n-1}|D| \over h\e}\suml_{i\in \I} \langle u\e \rangle_{B_i^{\eps}} \left(-\langle u\e \rangle_{\checkD}+\langle u\e \rangle_{\hatD}-\langle u\e \rangle_{B_i^{\eps}}\right)\\= q\e r\e \suml_{i\in \I} \langle u\e \rangle_{B_i^{\eps}} \left(-\langle u\e \rangle_{\checkD}+\langle u\e \rangle_{\hatD}-\langle u\e \rangle_{B_i^{\eps}}\right)\eps^{n-1}.
\end{multline}
Then, using the Cauchy inequality, condition \eqref{ass2}, estimates \eqref{est1}, \eqref{est2}, \eqref{1} and the fact that $u_1=0$ (and hence $\|u\e\|_{L_2(\Gamma)}\to 0$ as $\eps\to 0$), we obtain  from \eqref{delta1}:
\begin{multline}\label{delta2}
\left|(\nabla v\e,\nabla u\e-\nabla v\e)_{L_2(\Omega\e)}\right|^2\leq r\e(q\e)^2\left\{\suml_{i\in \I}\left|\langle u\e \rangle_{B_i^{\eps}}\right|^2 r\e\eps^{n-1}\right\}
\left\{\suml_{i\in\I}\left|\langle u\e \rangle_{\checkD}+\langle u\e \rangle_{\hatD}-\langle u\e \rangle_{B_i^{\eps}}\right|^2\eps^{n-1}\right\}\\\leq 
C_1 
\suml_{i\in\I}\left(\left|\langle u\e \rangle_{\Gamma_i^{\eps}}\right|^2\eps^{n-1}+\left|\langle u\e \rangle_{\checkD}-\langle u\e\rangle_{\Gamma_i^{\eps}}\right|^2\eps^{n-1}+\left|\langle  u\e \rangle_{\hatD}-\langle u\e \rangle_{B_i^{\eps}}\right|^2\eps^{n-1}\right)\\\leq C_2\left(\|u\e\|_{L_2(\Gamma)}^2+\eps^{n-1}\mathbf{D}\e\|\nabla u\e\|_{L_2(\cupl_{i\in\I}Y_i\e)}^2+\eps^{n-1}\mathbf{D}\e\|\nabla u\e\|_{L_2(\cupl_{i\in\I} B_i\e)}^2\right)\to 0\text{ as }\eps\to 0.
\end{multline}
Further, in view of \eqref{v1}, \eqref{g_est},
\begin{gather}\label{minmax4}
\liml_{\eps\to 0}(\nabla v\e,\nabla g\e)_{L_2(\Omega\e)}=0.
\end{gather}
And finally, using \eqref{normalization}, \eqref{v1} and \eqref{g_est}, we obtain:
\begin{gather}\label{minmax5}
\left|(\nabla g\e,\nabla w\e)_{L_2(\Omega\e)}\right|\leq \left|(\nabla g\e,\nabla u\e)_{L_2(\Omega\e)}\right|+\left|(\nabla g\e,\nabla v\e)_{L_2(\Omega\e)}\right|+\|\nabla g\e\|^2_{L_2(\Omega\e)}\to 0\text{ as }\eps\to 0.
\end{gather}
It follows from \eqref{minmax2}, \eqref{delta2}-\eqref{minmax5} that
\begin{gather}\label{delta3}
\liml_{\eps\to 0}(\nabla \tilde v\e,\nabla w\e)_{L_2(\Omega\e)}=0.
\end{gather}
Combining \eqref{minmax}, \eqref{minmax1} and \eqref{delta3} we conclude that
\begin{gather}\label{delta4}
\liml_{\eps\to 0}\|\nabla w\e\|^2_{L_2(\Omega\e)}=0
\end{gather}
and thus, in view of \eqref{repres}, \eqref{v1},  \eqref{g_est}, 
\eqref{delta4}, 
\textcolor{black}{we arrive at the required asymptotic:}
\begin{gather*}
\lambda\e = \|\nabla u\e\|^2_{L_2(\Omega\e)}\sim  \|\nabla v\e\|^2_{L_2(\Omega\e)}\sim q\text{ as }\eps\to 0.
\end{gather*}
\smallskip

\subsubsection{Proof of the property (B) of Hausdorff convergence} Let $\lambda\in \sigma(\mathcal{A}_{qr})$ if $r>0$ (respectively,  $\lambda\in \sigma({\mathcal{A}_q})$ if $r=0$). We have to prove that
\begin{gather}\label{Hb1}
\exists\lambda\e\in\sigma(\mathcal{A}\e):\ \lambda\e\to\lambda\text{ as }\eps\to 0. 
\end{gather}

Proving this indirectly we assume the opposite: \textcolor{black}{there are a
subsequence $\eps_k$, $\eps_k\underset{\eps\to 0}{\searrow}0$  and
a positive number $\delta$} such that
\begin{gather}
\label{dist} (\lambda-\delta,\lambda+\delta)\cap\sigma(\mathcal{A}^{\eps})=\varnothing\text{ as }\eps=\eps_k.
\end{gather}

Since $\lambda\in\sigma(\mathcal{A}_{qr})$ (respectively, $\lambda\in\sigma(\mathcal{A}_{q})$) there exists
$F=(f_1,f_2) \in L_2(\Omega)\oplus L_2(\Gamma)$, such that
\begin{gather}\label{notinim}
F\notin
\mathrm{im}(\mathcal{A}_{qr} -\lambda\mathrm{I})\
\text{(respectively, }F\notin\mathrm{im}(\mathcal{A}_q -\lambda\mathrm{I})\text{)}.
\end{gather}
We introduce the function $f\e\in\mathcal{H}\e$ by the formula
\begin{gather*}
f\e(x)=\begin{cases} f_1(x),&x\in\Omega,\\
0,& x\in\cupl_{i\in\I}T_i\e,\\
{1\over\sqrt{r\e}}{\langle \mathbf{f}_2 \rangle_{\Gamma_i\e}},&x\in
B_i\e,
\end{cases}
\end{gather*}
where $\mathbf{f}_2(x)=\sqrt{r} f_2(x)$ if $r>0$ (respectively, $\mathbf{f}_2(x)=f_2(x)$ if $r=0$).

One has:
\begin{gather*}
\|f\e\|^2_{\mathcal{H}\e}=\|f_1\|^2_{L_2(\Omega)}+{1\over r\e}\suml_{i\in\I}\varrho\e|B_i\e|\left|\langle \mathbf{f}_2 \rangle_{\Gamma_i\e}\right|^2\leq \|f_1\|^2_{L_2(\Omega)}+\|\mathbf{f}_2\|^2_{L_2(\Gamma)}\leq C\|F\|^2_{L_2(\Omega)\oplus L_2(\Gamma)}.
\end{gather*}

In view of \eqref{dist} $\lambda$ belongs to the resolvent set of $\mathcal{A}\e$ as $\eps=\eps_k$ and therefore  there exists the unique $u\e\in\mathrm{dom}(\mathcal{A}\e)$ such that
$$\mathcal{A}\e u\e-\lambda u\e=f\e,\ \eps=\eps_k$$
and moreover the following estimates are valid as $\eps=\eps_k$:
\begin{gather}\label{uL2+}
\|u\e\|_{\mathcal{H}\e}\leq \delta^{-1}\|f\e\|_{\mathcal{H}\e}\leq C_1,\\\label{uH1+}
\|\nabla u\e\|_{L_2(\Omega\e)}^2= \lambda\|u\e\|^2_{\mathcal{H}\e}+(f\e,u\e)_{\mathcal{H}\e}\leq C_2.
\end{gather}
It follows from \eqref{uL2+}-\eqref{uH1+} that there exist a subsequence $\eps_{k_l}\subset\eps_k$ and $u_1\in H^1(\Omega)$, $u_2\in L_2(\Gamma)$ such that
\eqref{conv1}-\eqref{conv4} hold (as $\eps=\eps_{k_l}\to 0$).

One has for an arbitrary $w\in H^1(\Omega\e)$:
\begin{gather}\label{int_eq_f}
\intl_{\Omega\e}\nabla u\e \cdot \nabla w  \d x -\lambda\intl_{\Omega\e}  u\e w \rho\e \d x=\intl_{\Omega\e}f\e w  \rho\e \d x,\ \eps=\eps_{k_l}.
\end{gather}
We plug into \eqref{int_eq_f} the function $w=w\e(x)$ defined by formula \eqref{we} and pass to the limit as $\eps=\eps_{k_l}\to 0$. In the same way as above we obtain that $(u_1,u_2)$ satisfies the equality
\begin{multline}\label{int_eq_f_final}
\intl_{\Omega}\nabla u_1\cdot \nabla w_1 \d x+\intl_{\Gamma}\left(qr u_1 w_1-q\sqrt{r}u_2 w_1-q\sqrt{r}u_1 w_2+q u_2 w_2\right)\d s-\lambda\left(\intl_{\Omega}u_1 w_1 \d x+\intl_{\Gamma}u_2 w_2 \d s\right)\\=\intl_{\Omega}f_1  w_1 \d x+\intl_{\Gamma}\mathbf{f}_2   w_2 \d s,
\end{multline}
which holds for an arbitrary $(w_1,w_2)\in C^\infty(\Omega)\oplus C^\infty(\Gamma)$ (and by the density arguments for an arbitrary $(w_1,w_2)\in H^1(\Omega)\oplus L_2(\Gamma)$).
It follows easily from \eqref{int_eq_f_final} that 
\begin{gather*}
\begin{array}{llll}
\text{if }r>0\text{ then }&U=(u_1,r^{-1/2}u_2)\in\mathrm{dom}(\mathcal{A}_{qr})&\text{ and }&\mathcal{A}_{qr} U-\lambda U=F,\\
\text{if }r=0\text{ then }& U=(u_1,u_2)\in\mathrm{dom}(\mathcal{A}_q)&\text{ and }&\mathcal{A}_q  U-\lambda U=F.
\end{array}
\end{gather*}
We obtain a contradiction to \eqref{notinim}. Thus there is 
$\lambda\e\in\sigma(\mathcal{A}\e)$ such that $\liml_{\eps\to 
0}\lambda\e=\lambda$.

\subsubsection{\label{sss223} Proof of Lemma \ref{lm1}}

In the proof of Theorem \ref{th1} we use the fact that 
\begin{gather}\label{qinsigma}
q\in\sigma(\mathcal{A}_{qr}). 
\end{gather}
In this section we prove Lemma \ref{lm1} containing, in particular, the property \eqref{qinsigma}.

At first we study the point spectrum of the operator $\mathcal{A}_{qr}$. Let $\lambda\not=q$ be an eigenvalue of $\mathcal{A}_{qr}$ corresponding to the eigenfunction $U=(u_1,u_2)\not= 0$. It means that $\forall V=(v_1,v_2)\in H^1(\Omega)\oplus L_2(\Gamma)$
\begin{gather}
\label{disc1}
\intl_{\Omega}\nabla u_1\cdot\nabla \overline{v_1} \d x+qr\intl_{\Gamma}(u_1-u_2)(\overline{v_1-v_2})\d s=\lambda\left(\intl_{\Omega}u_1 \overline{v_1}\d x+\intl_{\Gamma}u_2 \overline{v_2} r \d s\right).
\end{gather}
One has $u_1\not= 0$  (otherwise, plugging $u_1=0$ into \eqref{disc1} we arrive at $u_2=0$, that contradicts to $U\not= 0$).
Moreover, it is straightforward to show that if $U=(u_1,u_2)$ satisfies \eqref{disc1} then 
$u_2={qu_1|_\Gamma\over q-\lambda}$ and $u_1$  satisfies
\begin{gather}
\label{disc2}
\intl_{\Omega}\nabla u_1\cdot\nabla \overline{v_1} \d x-{\lambda qr\over q-\lambda}\intl_{\Gamma}u_1 \overline{v_1} \d s=\lambda\intl_{\Omega}u_1 \overline{v_1} \d x,\quad \forall v_1\in H^1(\Omega).
\end{gather}
Conversely if $u_1\in H^1(\Omega)$ satisfies \eqref{disc2} then $U=(u_1,u_2)$, where $u_2={qu_1|_\Gamma\over q-\lambda}$, satisfies \eqref{disc1}.

Let $\mu\in\mathbb{R}$. By $\eta^\mu$ we denote the sesquilinear form in $L_2(\Omega)$ defined as follows:
$$\eta^\mu[u,v]=\intl_{\Omega}\nabla u\cdot\nabla \bar v \d x-\mu \intl_\Gamma u\bar v \d s,\quad \mathrm{dom}(\eta^\mu)=H^1(\Omega).$$
We denote by $\mathcal{A}^\mu$ the operator generated by this form.
Formally the eigenvalue problem $\mathcal{A}^\mu u=\lambda u$ can be written as
\begin{gather}\label{robin}
-\Delta u=\lambda u\text{ in }\Omega,\quad {\partial u\over\partial  n}=\mu u\text{ on }\Gamma,\quad {\partial u\over\partial  n}=0\text{ on }\partial\Omega\setminus\Gamma,
\end{gather}
i.e.  $\mathcal{A}^\mu$ is the Laplacian in $\Omega$ subject to the Robin boundary conditions on $\Gamma$ and the Neumann ones on $\partial\Omega\setminus\Gamma$.

The spectrum of $\mathcal{A}^\mu$ is purely discrete. We denote by 
$\{\lambda_k(\mu)\}_{k\in\mathbb{N}}$ the
sequence of eigenvalues of $\mathcal{A}^\mu$ written in the ascending order and repeated according to their multiplicity. 

We denote by $\sigma_{\mathrm{p}}(\mathcal{A}_{qr})$  the set of eigenvalues of $\mathcal{A}_{qr}$.
It follows from the arguments above that
\begin{gather}\label{lambdamu} 
\sigma_{\mathrm{p}}(\mathcal{A}_{qr})\setminus\{q\}=\left\{\lambda\in\mathbb{R}: 
\ \exists\ k\in\mathbb{N}\text{ such that }
\lambda=\lambda_k(\mu),\text{ where }\mu={\lambda qr\over q-\lambda}
\right\}.
\end{gather}

Using the minimax principle it not hard to prove the following well-known 
properties of 
the eigenvalues of $\mathcal{A}^\mu$:

\begin{itemize}
\item for each $k\in\mathbb{N}$ the function $\mu\mapsto \lambda_k(\mu)$ is continuous and monotonically decreasing,

\item for each $k\in\mathbb{N}$ $\lambda_k(\mu)\to \lambda_k^D$ as $\mu\to -\infty$, where 
$\lambda_k^D$ is the $k$-th eigenvalue of the operator $\mathcal{A}^D$ acting in $L_2(\Omega)$ and generated by the form
$$\eta^D[u,v]=\intl_{\Omega}\nabla u\cdot\nabla \bar v \d x,\quad\mathrm{dom}(\eta^D)=\{u\in H^1(\Omega):\ u=0\text{ on }\Gamma\}$$
(i.e. $\mathcal{A}^D$ is the Laplacian in $\Omega$ subject to the Dirichlet boundary conditions on $\Gamma$ and the Neumann ones on $\partial\Omega\setminus\Gamma$),

\item for each $k\in\mathbb{N}$ $\lambda_k(\mu)\to -\infty$ as $\mu\to\infty$.\footnote{For the fulfilment of this property it is essential that $n\geq 2$. In the case $n=1$ this property is violated. Namely, let us consider the problem\quad $-u''=\lambda u$ on $(0,T)$, $u'(T)=\mu u(T)$, $u'(0)=0$. Its first eigenvalue $\lambda_1(\mu)$ goes to $-\infty$ as $\mu\to \infty$, while for $k\geq 2$ $\lambda_k(\mu)$ goes to the $(k-1)$-th eigenvalue of the problem $-u''=\lambda u$ on $(0,T)$, $u(T)=0$, $u'(0)=0$.}

\end{itemize}

We denote by $\Upsilon$ the curve $$\Upsilon=\left\{(\lambda,\mu)\in\mathbb{R}^2: \mu={\lambda qr\over q-\lambda}\right\}.$$ It consists of two branches $\Upsilon_{\pm}=\left\{(\lambda,\mu)\in\Upsilon: \pm(q-\lambda)> 0\right\}$.  
We also introduce the curves $\Upsilon_k=\{(\lambda,\mu)\in\mathbb{R}^2:\ \lambda=\lambda_k(\mu)\}$, $k\in\mathbb{N}$.

From the properties above we deduce the following:  

\begin{itemize}
\item For each $k\in\mathbb{N}$ the curve $\Upsilon_k$ intersects
the branch $\Upsilon_{+}$ exactly in one point (we denote the corresponding value of $\lambda$ by $\lambda_k^+$) . 

\item  We denote by $k_0$ the smallest integer satisfying $\lambda_{k_0}^D\leq q$ and $\lambda_{k_0+1}^D> q$. Then for each $k\in\mathbb{N}$ the curve $\Upsilon_{k+k_0}$ intersects 
the branch $\Upsilon_{-}$ exactly in one point (we denote the corresponding value of $\lambda$ by $\lambda_k^-$). For $k\leq k_0$ the curve $\Upsilon_k$ has no intersections with $\Upsilon_-$.

\item \eqref{distr} holds true.
\end{itemize}
Thus, taking into account \eqref{lambdamu}, we conclude that
\begin{gather}\label{lambdamu+}
\sigma_{\mathrm{p}}(\mathcal{A}_{qr})\setminus\{q\}=\{\lambda_k^-,k=1,2,3...\}\cup\{\lambda_k^+,k=1,2,3...\}.
\end{gather}

Since $\lambda_k^+\nearrow q$ as $k\to\infty$ then $q\in\sigma_{\mathrm{ess}}(\mathcal{A}_{qr})$. 

It remains to prove that if $\lambda\notin  \mathcal{S}=: 
\left(\cupl_{k\in\mathbb{N}}\{\lambda_k^-\}\right)\cup\left(\cupl_{k\in\mathbb{N}}
\{\lambda_k^+\}\right)$ and $\lambda\not=q$ then $\lambda$ belongs to the resolvent set of $\mathcal{A}_{qr}$.
Namely, we have to show that 
the problem
\begin{gather}\label{resolvent}
\mathcal{A}_{qr}V-\lambda V=F
\end{gather}
has a solution $V$ for an arbitrary $F=(f_1,f_2)\in\mathcal{H}$.

For  $\nu\in\mathbb{R}$ we introduce the operator $\widehat{\mathcal{A}}^{\nu}$ ($\nu\in\mathbb{R}$) acting in $L_2(\Omega)\oplus L_2(\Gamma)$ and generated by the sesquilinear form
\begin{gather*}
\widehat{\eta}^{\nu}[U,V]=\intl_{\Omega}\nabla u_1\cdot\nabla \overline{v_1} \d x-
\nu\intl_\Gamma u_1 \overline{v_1}\d s,\quad \mathrm{dom}(\widehat{\eta}^{\nu})=\left\{U=(u_1,u_2)\in H^1(\Omega)\oplus L_2(\Gamma):\ u_1|_\Gamma=u_2\right\}.
\end{gather*} 
The resolvent equation $\widehat{\mathcal{A}}_{\nu} U-\lambda U=G$ 
(where $U=\left(u,u|_\Gamma\right)$, $G=\left(g_1,g_2\right)$) formally can be written as follows:
\begin{gather*}
\begin{cases}
-\Delta u-\lambda u=g_1&\text{ in }\Omega,\\ 
{\partial u\over\partial  n}-\nu u -\lambda u=  g_2&\text{ on }\Gamma,\\{\partial u\over\partial n}=0&\text{ on }\partial\Omega\setminus \Gamma.
\end{cases}
\end{gather*}
The spectrum of $\widehat{\mathcal{A}}^{\nu}$ is purely discrete.

Let us consider the set
$$\widehat{\mathcal{S}}:=\left\{\lambda\in\mathbb{C}:\ \lambda\text{ is an eigenvalue of }\widehat{\mathcal{A}}_{\nu},\text{ where }\nu={qr\lambda\over q-\lambda}-\lambda\right\}.$$
Evidently, if $\lambda\not\in \widehat{\mathcal{S}}\cup\{q\}$ then for an arbitrary $G\in L_2(\Omega)\oplus L_2(\Gamma)$
the problem 
\begin{gather}\label{mures}
\widehat{\mathcal{A}}_{\nu}U-\lambda U=G,\text{ where }\nu={qr\lambda\over q-\lambda}-\lambda
\end{gather}
has a solution.

One can easily see that $\lambda$ belongs to $\widehat{\mathcal{S}}$ if and only if
$\lambda$ is an eigenvalue of the operator $\mathcal{A}^\mu$, where $\mu={qr\lambda\over q-\lambda}$. Using this and \eqref{lambdamu}-\eqref{lambdamu+} we conclude that 
$\widehat{\mathcal{S}}=\mathcal{S}$. 

Thus, if $\lambda\not\in \mathcal{S}\cup\{q\}$ then for an arbitrary 
$G=(g_1,g_2)\in L_2(\Omega)\oplus L_2(\Gamma)$ the problem \eqref{mures} has a 
solution $U=(u_1,u_2)$ (recall, that $u_1|_\Gamma=u_2$).
Then we take $g_1:=f_1$, $g_2:={qrf_2\over q-\lambda}$. It is straightforward to show that 
$V:=(u_1,{q\over q-\lambda}u_2)$ is a solution of  \eqref{resolvent}.

Obviously, all eigenvalues $\lambda_k^\pm$ have finite multiplicity and are isolated 
points of $\sigma(\mathcal{A}_{qr})$, whence 
$\sigma_{\mathrm{p}}(\mathcal{A}_{qr})\setminus\{q\}=
\sigma_{\mathrm{disc}}(\mathcal{A}_{qr})$.

Lemma \ref{lm1} is proved.

\subsection{\label{subsec23} Proof of Theorem \ref{th1}: the case $q=\infty $}

We prove the property \eqref{ah} of the Hausdorff convergence. 
Let $\lambda\e\in\sigma(\mathcal{A}\e)$ and $\lambda\e\to \lambda$ as $\eps\to 0$, we have to show that either 
\begin{gather}\label{HausA2}
\lambda\in\sigma(\mathcal{A}_r)\text{ if }r>0\text{\quad or\quad  }\lambda\in\sigma(\mathcal{A})\text{ if }r=0.
\end{gather}

Again by $u\e$ we denote the eigenfunction corresponding to $\lambda\e$ and satisfying \eqref{normalization}.
In the same way as in the case $q<\infty$ we conclude that  there is a subsequence (still denoted by $\eps$) and $u_1\in H^1(\Omega)$, $u_2\in L_2(\Gamma)$ such that \eqref{conv1}-\eqref{conv4} hold true. 

It is not hard to prove, using the trace inequality and the Poincar\'e inequality, 
the following estimate:
\begin{gather}\label{trP}
\|u-\<u\>_{\Gamma_i\e}\|^2_{L_2(\Gamma_i\e)}\leq C\eps \|\nabla u\|^2_{L_2(Y_i\e)},\ \forall u\in H^1(Y_i\e).
\end{gather}
Then, using \eqref{trP} and Lemmata \ref{lm2}-\ref{lm3}, we obtain:
\begin{multline}\label{u1u2+}
\liml_{\eps\to 0}\|\sqrt{r\e}u\e-\Pi\e u\e\|^2_{L_2(\Gamma)}=
\liml_{\eps\to 0}\suml_{i\in\I}\intl_{\Gamma_i\e}\left|\sqrt{r\e}u\e-\sqrt{r\e}\langle u\e\rangle_{B_i\e}\right|^2\d x\\\leq
C\liml_{\eps\to 0}\left(r\e\suml_{i\in\I}\intl_{\Gamma_i\e}\left|u\e-\langle u\e\rangle_{\Gamma_i\e}\right|^2\d x
+r\e\eps^{n-1}\suml_{i\in\I}\left|\langle u\e\rangle_{\Gamma_i\e}-\langle u\e\rangle_{\checkD}\right|^2\d x\right.\\+\left.
r\e\eps^{n-1}\suml_{i\in\I}\left|\langle u\e\rangle_{\checkD}-\langle u\e\rangle_{\hatD}\right|^2\d x+r\e\eps^{n-1}\suml_{i\in\I}\left|\langle u\e\rangle_{\hatD}-\langle u\e\rangle_{B_i\e}\right|^2\d x\right)\\\leq
C_1\liml_{\eps\to 0}\left(r\e\eps\suml_{i\in \I}\|\nabla u\e\|_{L_2(Y_i\e)}^2+
r\e\eps^{n-1}\mathbf{D}\e\suml_{i\in \I}\|\nabla u\e\|_{L_2(Y_i\e\cup B_i\e)}^2+
(q\e)^{-1}\suml_{i\in \I}\|\nabla u\e\|_{L_2(T_i\e)}^2\right)= 0,
\end{multline}
whence, in view of \eqref{conv3}-\eqref{conv4}, 
\begin{gather}\label{u1u2}
u_2=r^{1/2}u_1\text{ on }\Gamma. 
\end{gather}
Also one has, using the equality $\|\Pi\e u\e\|^2_{L_2(\Gamma)}=\suml_{i\in\I}\varrho\e |B_i\e|\cdot|\<u\e\>_{B_i\e}|^2$, 
\begin{gather}\label{1=}
1=\|u\e\|^2_{L_2(\Omega)}+
\suml_{i\in\I}\|u\e\|^2_{L_2(T_i\e)}+\|\Pi\e u\e\|^2_{L_2(\Gamma)}
+\suml_{i\in\I}\intl_{B_i\e}\varrho\e\left|u\e-\langle u\e\rangle_{B_i\e}\right|^2 \d x.
\end{gather}
Here the second term tends to zero in view of Lemma \ref{lm4}, the last term tends to zero in view of \eqref{uB} (the validity of \eqref{uB} is independent of either $q$ is finite or infinite). Thus, taking into account \eqref{u1u2}, we obtain from  \eqref{1=}:
\begin{gather*}
1=\|u_1\|^2_{L_2(\Omega)}+r\|u_1\|^2_{L_2(\Gamma)},
\end{gather*}
whence, 
\begin{gather}\label{u0}
u_1\not=0.
\end{gather}

For an arbitrary $w\in H^1(\Omega\e)$ one has equality \eqref{int_eq}. This time we choose the test-function $w$ as follows:
\begin{gather}\label{we1}
w(x)=\widetilde w\e(x):=
\begin{cases}
w(x)+\ds\suml_{i\in \I}(w(x^{i,\eps})-w(x))\varphi\e(x),& x\in\Omega,\\
w(x^{i,\eps}),&x\in T_i\e\cup B_i\e.
\end{cases}
\end{gather}
Here $w\in C^\infty(\Omega)$ in an arbitrary function, the cut-off function $\varphi\e(x)$ is the same as in \eqref{we}. 

We plug $w=\widetilde w\e(x)$  into \eqref{int_eq} and pass to the limit as $\eps\to 0$. 
Using \eqref{conv1} we obtain
\begin{gather}\label{LHS+}
\intl_{\Omega\e}\nabla u\e \cdot\nabla \widetilde w\e\d x=\intl_{\Omega}\nabla u\e\cdot \nabla w \d x+\suml_{i\in\I}\intl_{Y_i\e}\nabla u\e\cdot\nabla \left(\left(w(x^{i,\eps})-w\right)\phi_i\e\right)\d x\underset{\eps\to 0}\to 
\intl_{\Omega}\nabla u \cdot\nabla w \d x
\end{gather}
(the second integral vanishes because of the same arguments as those ones in the proof of \eqref{Iomega2}).

Now, let us study the right-hand-side of \eqref{int_eq}. We have:
\begin{multline}\label{Jomega+}
\lambda\e\intl_{\Omega\e} u\e \widetilde w\e \rho\e \d x=\lambda\e\left(\intl_{\Omega}u\e w \d x+\suml_{i\in\I}\intl_{Y_i\e} u\e \left(w(x^{i,\eps})-w\right)\phi_i\e \d x\right.\\\left.+\suml_{i\in\I}\intl_{T_i\e}u\e w(x^{i,\eps}) \d x+{\varrho\e}\suml_{i\in\I}\intl_{B_i\e} u\e w(x^{i,\eps})\d x\right).
\end{multline}

One has, using \eqref{conv2},
\begin{gather*}
\intl_{\Omega\e}u\e w  \d x\to \intl_{\Omega}u_1 w \d x\text{ as }\eps\to 0.
\end{gather*}
The second integral in \eqref{Jomega+} vanishes (here we use the same arguments as in \eqref{Jomega2}), the third integral also tends to zero as $\eps\to 0$: 
\begin{gather*}
\left|\suml_{i\in\I}\intl_{T_i\e}u\e w(x^{i,\eps}) \d x\right|^2\leq 
\suml_{i\in \I}\|u\e\|^2_{L_2(T_i\e)}\suml_{i\in\I}|T_i\e|\leq Ch\e \|u\e\|^2_{\mathcal{H}\e} \to 0\text{ as }\eps\to 0.
\end{gather*}

It remains to study the behaviour of the last integral in \eqref{Jomega+}. One has:
\begin{gather*}
{\varrho\e}\suml_{i\in\I}\intl_{B_i\e} u\e w(x^{i,\eps})\d x=\varrho\e (b\e)^n |B|
\suml_{i\in\I}\langle u\e\rangle_{B_i\e}w(x^{i,\eps})=
\sqrt{r\e}\intl_\Gamma \Pi\e u\e Q\e w \d s\underset{\eps\to 0}\to \sqrt{r}\intl_{\Gamma}u_2 w \d s.
\end{gather*}

Thus, taking into account \eqref{u1u2}, we conclude that 
\begin{gather}\label{RHS+}
\liml_{\eps\to 0}\left(\lambda\e\intl_{\Omega} u\e w\e \rho\e \d x\right)=\lambda\left(\intl_{\Omega}u_1 w \d x+\intl_{\Gamma}u_1 w r \d s\right).
\end{gather}

Combining \eqref{int_eq},  \eqref{LHS+} and \eqref{RHS+} we obtain:
\begin{gather}\label{int_eq_final+}
\intl_\Omega\nabla u_1\cdot\nabla w \d x= \lambda\left(\intl_{\Omega}u_1 w \d x+{r}\intl_{\Gamma}u_1 w \d s\right).
\end{gather}
Since $u_1\not=0$ then it follows easily from \eqref{int_eq_final+} that  
either $\lambda$ is the eigenvalue of $\mathcal{A}_r$ if $r>0$ or 
$\lambda$ is the eigenvalue of $\mathcal{A}$ if $r=0$. 
Thus the property \eqref{ah} of the Hausdorff convergence is proved.

The property \eqref{bh} of the Hausdorff convergence is proved in the same way as that one for the case $q<\infty$ (using the test-function $w=\tilde{w}\e(x)$ defined below by \eqref{we1}  instead of $w=w\e(x)$ defined by \eqref{we}). 

Theorem \ref{th1} is proved.

\section*{Acknowledgements}
\addcontentsline{toc}{section}{Acknowledgements}

G.C. is a member of GNAMPA of INdAM. The work of A.K. is supported by the German
Research Foundation (DFG) through Research Training Group 1294
"Analysis, Simulation and Design of Nanotechnological Processes".
Part of this work was done during the visit of A.K. to the Department of Engineering of University of Sannio (Benevento, Italy) and was partially supported by the project "Spectral theory and asymptotic analysis", FRA 2013.

\end{document}